\documentclass{article}

\usepackage
{
    graphicx,
    amssymb,
    amsmath,
    amsthm,
    dsfont, 
    xcolor,
    epstopdf,
    braket,
    mathtools,
    algpseudocode,
    authblk
}

\usepackage[colorlinks,
            pdffitwindow=false,
            plainpages=false,
            pdfpagelabels=true,
            pdfpagemode=UseOutlines,
            pdfpagelayout=SinglePage,
            bookmarks=false,
            colorlinks=true,
            hyperfootnotes=false,
            linkcolor=blue,
            citecolor=green!50!black]{hyperref}

\usepackage[hmargin=2cm,vmargin=2.5cm]{geometry}
\usepackage[bf]{caption}

\DeclareMathAlphabet{\mathpzc}{OT1}{pzc}{m}{it}

\newcommand{\subfiguretitle}[1]{{\scriptsize{#1}} \\[1mm]}
\newcommand{\R}{\mathbb{R}}                                     
\newcommand{\C}{\mathbb{C}}                                     
\newcommand{\innerprod}[2]{\left\langle #1,\, #2 \right\rangle} 
\newcommand{\textsub}[1]{\text{\tiny{#1}}}                      
\providecommand{\norm}[1]{\left\lVert #1 \right\rVert}          
\providecommand{\grad}{\nabla}                                  

\newcommand\xqed[1]{\leavevmode\unskip\penalty9999 \hbox{}\nobreak\hfill \quad\hbox{#1}}
\newcommand{\exampleSymbol}{\xqed{$\triangle$}}

\DeclareMathOperator{\rank}{rank}
\DeclareMathOperator{\vect}{vec}
\DeclareMathOperator{\ind}{ind}

\newtheorem{theorem}{Theorem}[section]

\newtheorem{proposition}[theorem]{Proposition}

\theoremstyle{definition}
\newtheorem{example}[theorem]{Example}

\makeatletter
\renewcommand*\env@matrix[1][*\c@MaxMatrixCols c]{%
  \hskip -\arraycolsep
  \let\@ifnextchar\new@ifnextchar
  \array{#1}}
\makeatother

\title{Towards tensor-based methods for the numerical approximation of the Perron--Frobenius and Koopman operator}
\author[1]{Stefan Klus}
\author[1,2]{Christof Sch\"utte}
\affil[1]{Department of Mathematics and Computer Science, Freie Universit\"at Berlin, Germany}
\affil[2]{Zuse Institute Berlin, Germany}
\date{}

\begin{document}
\maketitle

\begin{abstract}
The global behavior of dynamical systems can be studied by analyzing the eigenvalues and corresponding eigenfunctions of linear operators associated with the system. Two important operators which are frequently used to gain insight into the system's behavior are the Perron--Frobenius operator and the Koopman operator. Due to the curse of dimensionality, computing the eigenfunctions of high-dimensional systems is in general infeasible. We will propose a tensor-based reformulation of two numerical methods for computing finite-dimensional approximations of the aforementioned infinite-dimensional operators, namely Ulam's method and Extended Dynamic Mode Decomposition (EDMD). The aim of the tensor formulation is to approximate the eigenfunctions by low-rank tensors, potentially resulting in a significant reduction of the time and memory required to solve the resulting eigenvalue problems, provided that such a low-rank tensor decomposition exists. Typically, not all variables of a high-dimensional dynamical system contribute equally to the system's behavior, often the dynamics can be decomposed into slow and fast processes, which is also reflected in the eigenfunctions. Thus, the weak coupling between different variables might be approximated by low-rank tensor cores. We will illustrate the efficiency of the tensor-based formulation of Ulam's method and EDMD using simple stochastic differential equations.
\end{abstract}

\section{Introduction}

The Perron--Frobenius operator and the Koopman operator enable the analysis of the global behavior of dynamical systems. Eigenfunctions of these operators can be used to extract the dominant dynamics, to detect almost invariant sets, or to decompose the system into fast and slow processes \cite{DJ99, DFJ00, PDHSM04, BMM12, WKR14, FGH14a, FGH14b}. Assume the state space of your system is $ \R^d $ and you want to discretize each direction using $ n $ grid points (or boxes if Ulam's method is used), then overall $ n^d $ values need to be stored. Even for $ d = 10 $ and $ n = 10 $, more than $ 70 $ gigabyte of storage space would be required, whereas typical systems might have hundreds or thousands of dimensions and naturally also require a more fine-grained discretization. This so-called \emph{curse of dimensionality} can be overcome by using tensor formats which compress the data and store only the information that is relevant for the reconstruction. In general, only an approximation of the original data can be retrieved. Approximating the objects under consideration by sums of low-rank tensor products has become a powerful approach for tackling high-dimensional problems~\cite{HRS12} and in many physically significant problems near-linear complexity can be achieved since the separation rank depends only weakly on the dimension~\cite{BM05}. For high-dimensional systems exhibiting multiscale behavior, it might be possible to represent the weak coupling between different variables by low-rank tensor cores. The leading eigenfunctions of the Koopman operator, for instance, are typically almost constant for the fast variables of the system and depend mainly on the slowly changing variables (see \cite{FGH14a}). Thus, using tensor-based algorithms could reduce the amount of time and memory required to compute and store eigenfunctions significantly. In this way, analyzing high-dimensional systems that could not be tackled using standard methods might become feasible.

Tensors, in our sense, are just multidimensional arrays as shown in Figure~\ref{fig:Tensors}. Here and in what follows, standard vectors will be denoted by lower-case letters, e.g.~$ v $, matrices by upper-case letters, e.g.~$ A $, and tensors by the corresponding bold symbols, e.g.~$ \mathbf{x} $. It is important to note that tensors are typically not explicitly given -- for example by observed data --, but only implicitly as solutions of systems of linear or nonlinear equations or eigenvalue problems~\cite{GKT13}. Thus, numerical methods that operate directly on tensor approximations need to be developed since the full tensors cannot be stored or handled anymore in practice.

\begin{figure}[htb]
    \centering
    \includegraphics[width=0.5\textwidth]{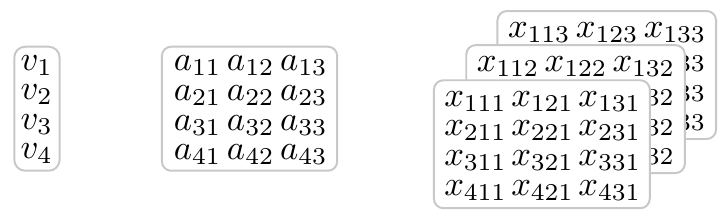}
    \caption{Tensors as multidimensional arrays. Here, $ v = (v_i) \in \R^4 $, $ A = (a_{ij}) \in \R^{4 \times 3} $, and $ \mathbf{x} = (x_{ijk}) \in \R^{4 \times 3 \times 3} $.}
    \label{fig:Tensors}
\end{figure}

Over the last years, low-rank tensor approximation approaches have become increasingly popular in the scientific computing community and are now becoming a standard tool to cope with large-scale problems that could not be handled before by standard numerical methods. An overview of different low-rank tensor approximation approaches can be found in \cite{GKT13}.

In this paper, we will show that the use of low-rank tensor approximation schemes potentially enables the computation of eigenfunctions of high-dimensional systems. The aim of this paper, however, is not to show that our approach is more efficient -- the tensor algorithms are mainly implemented in Matlab, a comparison with highly optimized numerical libraries implemented in C or C++ would not lead to meaningful results, developing high performance libraries for large-scale tensor problems is a separate task --, but to derive a tensor-based reformulation of existing methods and to show equivalency so that the theory available for the conventional matrix-vector based formulation can be carried over to multi-dimensional arrays. Furthermore, tensors could enable low-rank approximations of the Perron--Frobenius and Koopman operator as well as their eigenfunctions. One of the main future goals is to combine low-rank tensor decomposition techniques and the splitting of the dynamics into fast and slow processes. In \cite{FGH14a}, it has been shown that such a splitting of a multi-scale system exists and can be exploited to extract the slow dynamics. Another open problem is the generation of the low-rank approximations of the operators. Currently, the canonical tensor format representations are converted to the tensor-train format, which is time-consuming and in general leads to high ranks. Ideally, low-rank TT approximations should be directly generated from the given data.

We will start by introducing standard methods such as Ulam's method and the recently developed Extended Dynamic Mode Decomposition (EDMD) to approximate the eigenfunctions of the Perron--Frobenius operator and the Koopman operator in Section~\ref{sec:Perron--Frobenius and Koopman operator approximation}. Then, in Section~\ref{sec:Tensor formats}, different tensor formats will be described. In Section~\ref{sec:Tensor-based approximation}, we will reformulate Ulam's method and EDMD as tensor-based methods. Section~\ref{sec:Eigenvalue problems} contains a brief summary of simple power iteration schemes for the resulting tensor-based (generalized) eigenvalue problems. Simple examples which illustrate the proposed approaches are shown in Section~\ref{sec:Examples}. In Section~\ref{sec:Conclusion}, we will conclude with a short summary and possible future work.

\section{Perron--Frobenius and Koopman operator approximation}
\label{sec:Perron--Frobenius and Koopman operator approximation}

In this section, we will briefly introduce the Perron--Frobenius operator $ \mathcal{P} $ and the Koopman operator $ \mathcal{K} $ as well as numerical methods to compute finite-dimensional approximations, namely Ulam's method and EDMD. The main difference between Ulam's method and EDMD is that the former uses indicator functions\footnote{Higher-order methods for the approximation of the Perron--Frobenius operator have been proposed in \cite{DiDuLi}.} for a given box discretization of the domain while the latter allows arbitrary ansatz functions such as monomials, Hermite polynomials, trigonometric functions, or radial basis functions. Although EDMD was primarily developed for the approximation of the Koopman operator, it can be used to compute eigenfunctions of the Perron--Frobenius operator as well~\cite{KKS15}. Analogously, Ulam's method can also be used to compute eigenfunctions of the Koopman operator.

\subsection{Perron--Frobenius and Koopman operator}

Let $ S : \mathcal{X} \to \mathcal{X} $ be a dynamical system defined on a domain $ \mathcal{X} $, for example $ \mathcal{X} \subseteq \R^d $. Then the Perron--Frobenius operator or transfer operator $ \mathcal{P} $ is defined by
\begin{equation} \label{eq:PF operator}
    \int g \cdot \mathcal{P} f \, dm = \int (g \circ S) \cdot f \, dm,
\end{equation}
for all $ f, g \in \mathcal{F} $, where $ \mathcal{F} $ is an appropriately defined function space and $ \circ $ denotes function composition. We assume in what follows that $ \mathcal{F} = L^2(\mathcal{X}) $. The aim is to compute eigenfunctions of the Perron--Frobenius operator, given by
\begin{equation*}
    \mathcal{P} \varphi_i = \lambda_i \varphi_i.
\end{equation*}
The eigenfunction $ \varphi_1 $ corresponding to $ \lambda_1 = 1 $ is the invariant density of the system, i.e.~$ \mathcal{P} \varphi_1 = \varphi_1 $. The magnitude of the second largest eigenvalue $ \lambda_2 $ can be interpreted as the rate at which initial densities converge to the invariant density (for more details and assumptions about the dynamical system, see e.g.~\cite{FGH14a} and references therein). More generally, the leading eigenvalues of the Perron--Frobenius operator close to one correspond to the slowly converging transients of the system. The Koopman operator $ \mathcal{K} $, on the other hand, is defined by
\begin{equation*}
    \mathcal{K} f = f \circ S
\end{equation*}
and acts on functions $ f : \mathcal{X} \to \C $, $ f \in \mathcal{F} $. Correspondingly, the stochastic Koopman operator is defined by $ \mathcal{K} f = \mathbb{E}[\, f \circ S \,] $, where $ \mathbb{E}[\,\cdot\,] $ denotes the expected value with respect to the probability measure underlying $ S(x) $. We will only introduce the required notation and focus mainly on discrete-time dynamical systems, for more details on the Koopman operator and its properties, we refer to~\cite{BMM12, WKR14, WRK14}. While the Perron--Frobenius operator describes the evolution of \emph{densities}, the Koopman operator describes the evolution of \emph{observables}, which could be measurements or sensor probes~\cite{BMM12}. Instead of analyzing orbits $ \{x, \, S(x), \, S^2(x), \,  \dots \} $ of the dynamical system, we now analyze the measurements $ \{f(x), \, f(S(x)), \, f(S^2(x)), \, \dots \} $ at these points.

The Koopman operator $ \mathcal{K} $ is the adjoint of the Perron--Frobenius operator $ \mathcal{P} $ and thus an infinite-dimensional but linear operator. A finite-dimensional approximation (computed using generalized Galerkin methods) of this operator captures the dynamics of a nonlinear dynamical system without necessitating a linearization around a fixed point~\cite{BMM12, WKR14}. We are again interested in eigenfunctions of the operator, given by
\begin{equation*}
    \mathcal{K} \varphi_i = \lambda_i \varphi_i.
\end{equation*}
Let $ f : \mathcal{X} \to \R $ be an observable of the system that can be written as a linear combination of the linearly independent eigenfunctions $ \varphi_i $, i.e.
\begin{equation*}
    f(x) = \sum_i c_i \varphi_i(x),
\end{equation*}
with $ c_i \in \C $. Then
\begin{equation*}
    (\mathcal{K}f)(x) = \sum_i \lambda_i c_i \varphi_i(x).
\end{equation*}
Analogously, for vector-valued functions $ F = [f_1, \, \dots, \, f_n]^T $, we obtain
\begin{equation*}
    \mathcal{K} F =
    \begin{bmatrix}
        \sum_i \lambda_i c_{i, 1} \varphi_i \\
        \vdots \\
        \sum_i \lambda_i c_{i, n} \varphi_i
    \end{bmatrix} =
    \sum_i \lambda_i \varphi_i
    \begin{bmatrix}
        c_{i, 1} \\
        \vdots \\
        c_{i, n}
    \end{bmatrix} =
    \sum_i \lambda_i \varphi_i v_i,
\end{equation*}
where $ v_i = [c_{i, 1}, \, \dots, \, c_{i, n}]^T $. These vectors $ v_i $ corresponding to the eigenfunctions $ \varphi_i $ are called Koopman modes.

\subsection{Ulam's method}

A frequently used method to compute an approximation of the Perron--Frobenius operator is Ulam's method, see e.g.~\cite{DFJ00, CU02, BS13, FGH14a}. First, the state space $ \mathcal{X} $ is covered by a finite number of disjoint boxes $ \{ \mathcal{B}_1, \, \dots, \, \mathcal{B}_k \} $. Let $ \mathds{1}_{\mathcal{B}_i} $ be the indicator function for box $ \mathcal{B}_i $, i.e.
\begin{equation*}
    \mathds{1}_{\mathcal{B}_i}(x)
        = \begin{cases} 1,& \text{if } x \in \mathcal{B}_i, \\ 0,& \text{otherwise}. \end{cases}
\end{equation*}
Then a finite-dimensional approximation of the operator can be obtained as follows: Using definition \eqref{eq:PF operator} leads to
\begin{equation*}
    \int \mathds{1}_{\mathcal{B}_j} \cdot \mathcal{P} \mathds{1}_{\mathcal{B}_i} \, dm
        = \int (\mathds{1}_{\mathcal{B}_j} \circ S) \cdot \mathds{1}_{\mathcal{B}_i} \, dm
        = \int \mathds{1}_{S^{-1}(\mathcal{B}_j)} \cdot \mathds{1}_{\mathcal{B}_i} \, dm
        = m(S^{-1}(\mathcal{B}_j) \cap \mathcal{B}_i).
\end{equation*}
This relationship can be represented by a matrix $ \hat{P} = (\hat{p}_{ij}) \in \R^{k \times k} $ with
\begin{equation*}
    \hat{p}_{ij} = \frac{m\left(S^{-1}(\mathcal{B}_j) \cap \mathcal{B}_i\right)}{m(\mathcal{B}_i)}.
\end{equation*}
Here, in order to avoid confusion with the approximation $ \mathbf{P} $ on a tensor space introduced below, we denote the matrix representation $ \hat{P} $ instead of $ P $. The denominator $ m(\mathcal{B}_i) $ normalizes the entries $ \hat{p}_{ij} $ so that $ \hat{P} $ becomes a row-stochastic matrix and defines a finite Markov chain. The left eigenvector corresponding to the eigenvalue $ \lambda_1 = 1 $ approximates the invariant measure of the Perron--Frobenius operator $ \mathcal{P} $.

The entries $ \hat{p}_{ij} $ of the matrix $ \hat{P} $ represent the probabilities of points being mapped from box $ \mathcal{B}_i $ to box $ \mathcal{B}_j $ by the dynamical system $ S $. These entries can be estimated by randomly choosing a large number of test points $ x_i^{(l)} $, $ l = 1, \dots, n $, in each box $ \mathcal{B}_i $ and by counting how many times test points were mapped from box $ \mathcal{B}_i $ to box $ \mathcal{B}_j $ by $ S $, i.e.
\begin{equation} \label{eq:Ulam p_ij}
    \hat{p}_{ij} \approx \frac{1}{n} \sum_{l=1}^{n}
        \mathds{1}_{\mathcal{B}_j}\left(S\left(x_i^{(l)}\right)\right).
\end{equation}
The eigenfunctions of the Perron--Frobenius operator are then approximated by the left eigenvectors of the matrix $ \hat{P} $, the eigenfunctions of the Koopman operator by the right eigenvectors.

\subsection{Extended dynamic mode decomposition}

An approximation of the Koopman operator, the Koopman eigenvalues, eigenfunctions, and eigenmodes can be computed using EDMD. The method requires data, i.e.~a set of values $ x_i $ and the corresponding $ y_i = S(x_i) $ values, $ i = 1, \dots, m $, written in matrix form as
\begin{equation*}
    X =
    \begin{bmatrix}
        x_1 & x_2 & \cdots & x_m
    \end{bmatrix}
    \quad \text{and} \quad
    Y =
    \begin{bmatrix}
        y_1 & y_2 & \cdots & y_m
    \end{bmatrix},
\end{equation*}
and additionally a set of ansatz functions or observables
\begin{equation*}
    \mathcal{D} = \left\{ \psi_1, \, \psi_2, \, \dots, \, \psi_k \right\},
\end{equation*}
with $ \psi_i : \mathcal{X} \to \R $. Thus, $ X, Y \in \R^{d \times m} $. The vectors $ x_i $ are used as collocation points to approximate the integrals required for the approximation of the Koopman operator. Let
\begin{equation} \label{eq:Psi}
    \Psi =
    \begin{bmatrix}
        \psi_1 & \psi_2 & \cdots & \psi_k
    \end{bmatrix}^T,
\end{equation}
$ \Psi : \mathcal{X} \to \R^k $, be the vector of all ansatz functions, then $ \mathcal{K} $ can be approximated by a matrix $ \hat{K} \in \R^{k \times k} $, with
\begin{equation*}
    \hat{K}^T = \hat{A} \hat{G}^+,
\end{equation*}
where $ ^+ $ denotes the pseudoinverse. The matrices $ \hat{A}, \, \hat{G} \in \R^{k \times k} $ are defined as
\begin{equation} \label{eq:A and G entries}
    \begin{split}
        \hat{A} &= \frac{1}{m} \sum_{l=1}^m \Psi(y_l) \Psi(x_l)^T, \\
        \hat{G} &= \frac{1}{m} \sum_{l=1}^m \Psi(x_l) \Psi(x_l)^T.
    \end{split}
\end{equation}
As before, we use the $ \hat{\phantom{a}} $ symbol to distinguish the matrices from the tensor approximations that will be introduced in Section~\ref{sec:Tensor-based approximation}. An approximation of the eigenfunction $ \varphi_i $ of the Koopman operator $ \mathcal{K} $ is then given by
\begin{equation*}
    \varphi_i = \xi_i \Psi,
\end{equation*}
where $ \xi_i $ is the $ i $-th left eigenvector of the matrix $ \hat{K}^T $. Alternatively, the generalized eigenvalue problem
\begin{equation} \label{eq:EDMD generalized eig K}
    \xi_i \hat{A} = \lambda_i \xi_i \hat{G}
\end{equation}
can be solved, provided that $ \hat{G} $ is regular. To compute eigenfunctions of the Perron--Frobenius operator, the corresponding eigenvalue problem
\begin{equation*}
    \xi_i \hat{A}^T = \lambda_i \xi_i \hat{G}
\end{equation*}
needs to be solved. Note that this formulation is similar to the variational approach to compute eigenfunctions of transfer operators of reversible processes presented in~\cite{NKPMN14, NSVN15}. For more details, we refer the reader to~\cite{KKS15}.

\section{Tensor formats}
\label{sec:Tensor formats}

Several different tensor formats have been developed in the past, e.g.\ the canonical format, the Tucker format, and the tensor-train format. In this section, we will briefly introduce tensors and the required notation. The overall goal is to rewrite the methods presented in the previous section as tensor-based methods and to take advantage of low-rank tensor approximations and the fact that the dynamics of high-dimensional systems can often be decomposed.

\subsection{Full format}

A tensor in full format is simply a multidimensional array $ \mathbf{v} \in \R^{k_1 \times \dots \times k_d} $. (A variation of this format is the sparse format which stores only the nonzero entries and is used, for example, in the sparse grid approach~\cite{Hac14}.) The entries of a tensor $ \mathbf{v} $ are indexed by $ \mathbf{v}_\mathbf{i} = \mathbf{v}_{i_1, \dots, i_d} = \mathbf{v}[\mathbf{i}] = \mathbf{v}[i_1, \dots, i_d] $, where $ \mathbf{i} = (i_1, \dots, i_d) $ is a multi-index. Addition and subtraction are trivially defined element-wise. Multiplication of a tensor $ \mathbf{v} $ by a scalar $ c \in \R $ is naturally generalized as $ (c \mathbf{v})[i_1, \dots, i_d] = c \mathbf{v}[i_1, \dots, i_d] $. Matrix-vector multiplication is defined as follows: Given a linear operator $ \mathbf{A} $ defined on a tensor space $ \R^{k_1 \times \cdots \times k_d} $, with
\begin{equation*}
    \mathbf{A} = \mathbf{A}[i_1, \dots, i_d, j_1, \dots, j_d] \in \R^{k_1 \times \cdots \times k_d \times k_1 \times \cdots \times k_d},
\end{equation*}
the product of $ \mathbf{A} $ and $ \mathbf{v} $ is
\begin{equation*}
    (\mathbf{A} \mathbf{v})[i_1, \dots, i_d] = \sum_{j_1=1}^{k_1} \dots \sum_{j_d=1}^{k_d}
        \mathbf{A}[i_1, \dots, i_d, j_1, \dots, j_d]
        \mathbf{v}[j_1, \dots, j_d]
\end{equation*}
or in shorthand notation, using multi-indices,
\begin{equation*}
    (\mathbf{A} \mathbf{v})_\mathbf{i} = \sum_{\mathbf{j}}
        \mathbf{A}_\mathbf{ij} \mathbf{v}_\mathbf{j}.
\end{equation*}
Furthermore, the inner product of two tensors $ \mathbf{v}, \mathbf{w} \in \R^{k_1 \times \dots \times k_d} $ is defined as 
\begin{equation*}
    \innerprod{\mathbf{v}}{\mathbf{w}}
        = \sum_{i_1=1}^{k_1} \dots \sum_{i_d=1}^{k_d}
          \mathbf{v}[i_1, \dots, i_d] \mathbf{w}[i_1, \dots, i_d]
\end{equation*}
and the outer product $ \mathbf{v} \otimes \mathbf{w} \in \R^{k_1 \times \dots \times k_d \times k_1 \times \dots \times k_d} $ as a tensor with entries
\begin{equation*}
    (\mathbf{v} \otimes \mathbf{w})[i_1, \dots, i_d, j_1, \dots, j_d] =
        \mathbf{v}[i_1, \dots, i_d] \mathbf{w}[j_1, \dots, j_d].
\end{equation*}
The outer product $ \mathbf{v} \otimes \mathbf{w} $ can be regarded as a linear map that acts on tensors $ \R^{k_1 \times \dots \times k_d} $. Often it is required or convenient to rewrite a tensor as a vector. The \emph{vectorization} of a tensor, denoted $ \vect(\mathbf{v}) $, where $ \vect : \R^{k_1 \times \dots \times k_d} \to \R^{k_1 \cdots k_d} $, reorders the entries of $ \mathbf{v} $ into one column vector. For $ \mathbf{v} \in \R^{2 \times 3 \times 2} $, for example,
\begin{equation*}
    \vect(\mathbf{v}) =
    \begin{bmatrix}[cc|cc|c|cc|cc]
            v_{111} &
            v_{211} &
            v_{121} &
            v_{221} &
            \dots  &
            v_{122} &
            v_{222} &
            v_{132} &
            v_{232}
    \end{bmatrix}^T.
\end{equation*}

For our purposes, we will mainly be interested in eigenvalue problems of the form
\begin{equation*}
    \mathbf{A} \mathbf{v} = \lambda \mathbf{v}
    \quad \text{or} \quad
    \mathbf{A} \mathbf{v} = \lambda \mathbf{B} \mathbf{v}.
\end{equation*}
The treatment of tensors in the full format often leads to storage problems. Thus, different formats have been developed to overcome this problem. Instead of working with the full format, we will use compressed formats such as the $ r $-term or TT format for numerical computations in order to minimize computational costs as well as storage requirements. Except for very particular examples, it is impossible to compress the data without any compression error~\cite{Hac14}. Typically, the tensor representation is just an approximation of the original data. Below, we will describe different compressed tensor formats, the introduction is based on~\cite{Hac14}.

\subsection{Canonical format}

A tensor space is given by $ \mathbf{V} = \bigotimes_{\mu=1}^d V_\mu $, where $ V_1, \dots, V_d $ are vector spaces defined over the same field $ \mathbb{K} $, typically $ \R $ or $ \C $. An \emph{elementary tensor} is defined to be a product of the form
\begin{equation*}
    \mathbf{v} = v_1 \otimes \cdots \otimes v_d,
\end{equation*}
with $ v_\mu \in V_\mu $, $ \mu = 1, \dots, d $. An \emph{algebraic tensor} is then a linear combination of elementary tensors, i.e.
\begin{equation*}
    \mathbf{v} = \sum_{l=1}^r v_1^{(l)} \otimes \cdots \otimes v_d^{(l)}.
\end{equation*}
This format is also called $ r $-term format or CP format. That is, instead of trying to store the tensor in the dense format, one only considers tensors that can be written as products of the form $ \mathbf{v}[i_1, \dots, i_d] = v_1[i_1] \cdots v_d[i_d] $. If the best approximation of this form is not good enough, then the natural extension is to consider $ \mathbf{v}[i_1, \dots, i_d] = \sum_{l=1}^r v_1^{(l)}[i_1] \cdots v_d^{(l)}[i_d] $, cf.~\cite{BM02, BM05}. Defining
\begin{equation*}
    \mathcal{R}_r = \left\{ \sum_{l=1}^r v_1^{(l)} \otimes \cdots \otimes v_d^{(l)} : v_\mu^{(l)} \in V_\mu \right\}
\end{equation*}
for $ r \in \mathbb{N}_0 $, which implies that $ \{ 0 \} = \mathcal{R}_0 \subset \mathcal{R}_1 \subset \dots \subset \mathbf{V} $, we call
\begin{equation*}
    \rank(\mathbf{v}) = \min \{ r : \mathbf{v} \in \mathcal{R}_r \}
\end{equation*}
the \emph{tensor rank} of $ \mathbf{v} $.

\begin{example}
Let us consider a system of the form $ dx_t = -\grad_x V(x_t) \, dt + \sigma \, dW_t $, where $ V $ is the energy landscape associated with the system. The invariant density of this process is given by $ \mu(x) = \frac{1}{Z} e^{-\beta V(x)} $, where $ \beta $ and $ Z $ are constants \cite{SS13}. Assume that the state space is three-dimensional and that the potential function can be written as $ V(x) = V_1(x_1) + V_2(x_2) + V_3(x_3) $, then
\begin{equation*}
    \mu(x) = \frac{1}{Z} e^{-\beta V_1(x_1)} e^{-\beta V_2(x_2)} e^{-\beta V_3(x_3)}.
\end{equation*}
Thus, storing the invariant density for a grid with $ k $ grid points in each direction in the full format would require an array of size $ k^3 $ while storing it in the canonical tensor format would require only a tensor of rank $ 1 $ and thus an array of size $ 3 k $. \exampleSymbol
\end{example}

Given tensors in the $ r $-term format, basic operations are defined as follows~\cite{Hac14}:
\begin{itemize}
\item \textbf{Addition:}
\begin{equation*}
    \mathbf{v} = \sum_{l=1}^{r_v} \bigotimes_{\mu=1}^d v_\mu^{(l)}, \quad
    \mathbf{w} = \sum_{l=1}^{r_w} \bigotimes_{\mu=1}^d w_\mu^{(l)} \quad \Rightarrow \quad
    \mathbf{x} = \mathbf{v} + \mathbf{w} = \sum_{l=1}^{r_v+r_w} \bigotimes_{\mu=1}^d x_\mu^{(l)},
\end{equation*}
where
\begin{equation*}
    x_\mu^{(l)} =
    \begin{cases}
        v_\mu^{(l)}, & 1 \le l \le r_v, \\
        w_\mu^{(l-r_v)}, & r_v+1 \le l \le r_v+r_w.
    \end{cases}
\end{equation*}
\item \textbf{Matrix-vector multiplication:}
\begin{equation*}
    \mathbf{A} = \sum_{l_A=1}^{r_A} \bigotimes_{\mu=1}^d A_\mu^{(l_A)}, \quad
    \mathbf{v} = \sum_{l_v=1}^{r_v} \bigotimes_{\mu=1}^d v_\mu^{(l_v)} \quad \Rightarrow \quad
    \mathbf{A} \mathbf{v} = \sum_{l_A=1}^{r_A} \sum_{l_v=1}^{r_v} \bigotimes_{\mu=1}^d A_\mu^{(l_A)} v_\mu^{(l_v)}.
\end{equation*}
\end{itemize}

Since these operations increase the rank, truncation is typically required to approximate the resulting tensor $ \mathbf{v} $ with rank $ r_v $ by a tensor $ \tilde{\mathbf{v}} $ with a lower rank $ r_{\tilde{v}} $ by either fixing the rank $ r_{\tilde{v}} $ or by fixing $ \varepsilon $ such that $ \norm{ \mathbf{v} - \tilde{\mathbf{v}} } < \varepsilon $.

\subsection{TT format}

Another frequently used tensor format is the TT format -- TT now stands for \emph{tensor train} instead of the former \emph{tree tensor} --, which can be obtained by successive singular value decompositions. This is a special case of a more general hierarchical format and has been introduced in quantum physics under the name \emph{Matrix Product States} (MPS), see \cite{Hac14} for details. A tensor $ \mathbf{v} \in \R^{k_1 \times \dots \times k_d} $ is decomposed into $ d $ component tensors $ \mathbf{v}_i $ of at most order three (the first and last are of order two and are often, for the sake of simplicity, considered as tensors of order three with ``boundary condition'' $ \rho_0 = \rho_d = 1 $). That is, the entries of $ \mathbf{v} $ are given by
\begin{equation*}
    \mathbf{v}[i_1, \dots, i_d] = \sum_{k_1=1}^{\rho_1} \dots \sum_{k_{d-1}=1}^{\rho_{d-1}}
        \mathbf{v}_1[1, i_1, k_1] \, \mathbf{v}_2[k_1, i_2, k_2] \, \dots \, \mathbf{v}_{d-1}[k_{d-2}, i_{d-1}, k_{d-1}] \, \mathbf{v}_d[k_{d-1}, i_d, 1].
\end{equation*}
For fixed indices, the component tensors of rank three can be regarded as matrices, which leads to a more compact representation
\begin{equation*}
    \mathbf{v}[i_1, \dots, i_d] = V_1[i_1] \, V_2[i_2] \, \cdots \, V_{d-1}[i_{d-1}] \, V_d[i_d]
\end{equation*}
and justifies the original name MPS. Here, the numbers $ \rho_i $ are the ranks of the TT tensor, resulting in a rank vector $ \rho = [\rho_0, \rho_1, \dots, \rho_{d}] $ which determines the complexity of the representation. The main advantages of the TT format are its stability from an algorithmic point of view and reasonable computational costs, provided that the ranks of the tensors are small~\cite{HRS12}. The basic operations such as addition and matrix-vector multiplication are more complex than in the canonical format and can be found, for example, in \cite{Ose11}. Converting a tensor from the canonical format to the TT format is trivial, but the TT representation requires more memory. Numerical toolboxes for the TT decomposition of tensors and several algorithms for solving linear systems of equations are available online, see e.g.~\cite{Ose14}.

\subsection{Comparison}

Complexity-wise, the canonical format would be the ideal candidate for representing tensors since the number of required parameters depends only linearly on the dimension $ d $, the rank $ r $, and the sizes of the individual vector spaces. It turned out, however, that solving even simple problems using the canonical format is hard in practice due to redundancies and instabilities which can lead to numerical problems~\cite{HRS12}. The main advantage of the TT format is its structural simplicity, higher-order tensors are reduced to $ d $ products of tensors of at most order three. Similar approaches have been known in quantum physics for a long time, the rigorous mathematical analysis, however, is still work in progress (see \cite{HRS12} and references therein). For our purposes, we will rely on the TT format and the TT toolbox developed by Oseledets et al.~\cite{Ose14} and implement simple power iteration schemes to solve the resulting eigenvalue problems as we will show in Section~\ref{sec:Eigenvalue problems}. Other tensor formats, however, might be advantageous as well for the analysis of the Perron--Frobenius and Koopman operator. This should be investigated further in the future. One drawback of the TT format is that the decomposition depends on the ordering of the dimensions and thus results in different tensor ranks for different orderings.

\section{Tensor-based approximation}
\label{sec:Tensor-based approximation}

In this section, we will present a tensor-based reformulation of Ulam's method and EDMD and show that these methods are equivalent to the corresponding vector-based counterparts. The new formulation enables the use of the low-rank tensor approximation approaches described in the previous section. That is, variables can be approximated with different degrees of accuracy.

\subsection{Reformulation of Ulam's method}

Given a dynamical system $ S : \mathcal{X} \to \mathcal{X} $, $ \mathcal{X} \subset \R^d $, define a box $ \mathcal{B} $ that contains $ \mathcal{X} $, i.e. 
\begin{equation*}
    \mathcal{B} = \mathcal{I}_1 \times \cdots \times \mathcal{I}_d \supset \mathcal{X},
\end{equation*}
where each $ \mathcal{I}_\mu = [a_\mu, b_\mu] \subset \R $ is an interval. Furthermore, let each $ \mathcal{I}_\mu $ be partitioned into $ k_\mu $ subintervals $ \mathcal{I}_\mu^{i_\mu} $, $ i_\mu = 1, \dots, k_\mu $, such that $ \mathcal{I}_\mu^{i_\mu} \cap \mathcal{I}_\mu^{j_\mu} = \varnothing $ for $ i_\mu \ne j_\mu $. This results in a partitioning of $ \mathcal{B} $ into $ \hat{k} = \prod_{\mu=1}^d k_\mu $ boxes. Using again standard multi-index notation, we will denote $ \mathbf{i} = (i_1, \dots, i_d) $. Equipped with the mapping
\begin{equation} \label{eq:Index mapping}
    \mathbf{i} = (i_1, \dots, i_d) \mapsto
        \hat{i} = 1 + \sum_{\mu = 1}^d \left( \prod_{\nu=1}^{\mu-1} k_\nu \right) (i_\mu - 1),
\end{equation}
each multi-index $ \mathbf{i} $ corresponds to a number $ \hat{i} \in \{ 1, \dots, \hat{k} \} $. This induces a canonical numbering of the boxes
\begin{equation*}
    \mathcal{B}_\mathbf{i} = \mathcal{I}_1^{i_1} \times \dots \times \mathcal{I}_d^{i_d}
                           = \mathcal{B}_{\hat{i}}
\end{equation*}
and the entries of tensors $ \mathbf{x} \in \R^{k_1 \times \dots \times k_d} $ such that $ \mathbf{x}_\mathbf{i} = x_{\hat{i}} $, where $ x = \vect(\mathbf{x}) $. Thus, with the aid of Ulam's method we could now generate the finite-dimensional representation of the Perron--Frobenius operator $ \hat{P} \in \R^{\hat{k} \times \hat{k}} $ as described in Section~\ref{sec:Perron--Frobenius and Koopman operator approximation}. Our goal, however, is to approximate the operator by a tensor $ \mathbf{P} \in \R^{k_1 \times \dots \times k_d \times k_1 \times \dots \times k_d} $. Note that the indicator function for the box $ \mathcal{B}_\mathbf{i} $ can be written as
\begin{equation} \label{eq:indicator product}
    \mathds{1}_{\mathcal{B}_\mathbf{i}}(x) = \prod_{\mu=1}^d \mathds{1}_{\mathcal{I}_\mu^{i_\mu}}(x_\mu)
                                           = \mathds{1}_{\mathcal{B}_{\hat{i}}}(x).
\end{equation}
That is, each $ d $-dimensional indicator function $ \mathds{1}_{\mathcal{B}_{\mathbf{i}}}(x) $ is now written as a product of $ d $ one-dimensional indicator functions $ \mathds{1}_{\mathcal{I}_\mu^{i_\mu}}(x_\mu) $.

\begin{example} \label{ex:TensorUlam}
Let us start with a simple example which illustrates the idea behind the tensor-based formulation. Consider the box discretization $ \{ \mathcal{B}_1, \dots, \mathcal{B}_9 \} $ of $ \mathcal{B} = [0, 3]^2 $ shown in Figure~\ref{fig:Grid example}(a). Thus, using Ulam's method, we would obtain $ 9 $ indicator functions $ \{ \mathds{1}_{\mathcal{B}_1}, \dots, \mathds{1}_{\mathcal{B}_9} \} $ and the matrix $ \hat{P} $ that approximates the Perron--Frobenius operator $ \mathcal{P} $ would be a row-stochastic $ (9 \times 9) $-matrix. An example of such a matrix is shown in Figure~\ref{fig:Grid example}(b), the underlying dynamical system is not relevant here. The goal now is to rewrite this matrix using tensors (cf.~Example~\ref{ex:TensorUlamRewritten}).

\begin{figure}[htb]
    \centering
    \begin{minipage}[t]{0.4\textwidth}
        \centering
        \subfiguretitle{(a)}
        \includegraphics[width=0.7\textwidth]{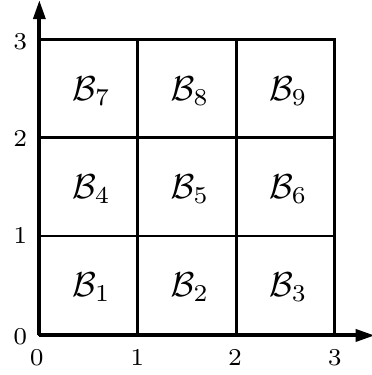}
    \end{minipage}
    \begin{minipage}[t]{0.59\textwidth}
    \centering
    \subfiguretitle{(b)} \vspace*{4ex}
        $ \hat{P} =
        \begin{bmatrix}
            0.68 & 0.09 & 0.07 & 0.04 & 0.02 &     0 & 0.09 & 0.01 &    0 \\ 
            0.36 & 0.06 & 0.40 & 0.03 &    0 &  0.02 & 0.05 & 0.03 & 0.05 \\
            0.07 & 0.12 & 0.64 & 0.02 &    0 &  0.07 & 0.03 &    0 & 0.05 \\
            0.31 & 0.07 & 0.02 & 0.09 & 0.01 &  0.02 & 0.39 & 0.05 & 0.04 \\
            0.25 & 0.05 & 0.18 & 0.06 & 0.01 &  0.02 & 0.17 & 0.05 & 0.21 \\
            0.06 & 0.06 & 0.37 & 0.01 &    0 &  0.04 & 0.07 & 0.03 & 0.36 \\
            0.17 & 0.01 & 0.01 & 0.09 & 0.01 &  0.01 & 0.60 & 0.05 & 0.05 \\
            0.05 &    0 & 0.06 & 0.08 &    0 &  0.05 & 0.29 & 0.13 & 0.34 \\
            0.01 &    0 & 0.03 & 0.02 & 0.02 &  0.11 & 0.05 & 0.09 & 0.67
        \end{bmatrix} $
    \end{minipage}
    \caption{(a) Box discretization of $ \mathcal{B} = [0, 3]^2 $. (b) Example of a resulting approximation $ \hat{P} $ of the Perron--Frobenius operator $ \mathcal{P} $, obtained by applying Ulam's method.}
    \label{fig:Grid example}
\end{figure}

Defining intervals $ \mathcal{I}_{\mu}^1 = [0, 1] $, $ \mathcal{I}_{\mu}^2 = [1, 2] $, and $ \mathcal{I}_{\mu}^3 = [2, 3] $ as well as indicator functions
\begin{equation*}
    \mathds{1}_{\mathcal{I}_\mu^{i_\mu}}(x_\mu) = 
    \begin{cases}
        1, & x_\mu \in \mathcal{I}_\mu^{i_\mu}, \\
        0, & \text{otherwise},
    \end{cases}
\end{equation*}
for $ \mu = 1, 2 $ and $ i_\mu = 1, 2, 3 $, the ansatz functions for the box discretization can be written as
\begin{alignat*}{2}
    \mathds{1}_{\mathcal{B}_{1, 1}}(x) 
        &= \mathds{1}_{\mathcal{I}_1^1}(x_1) \mathds{1}_{\mathcal{I}_2^1}(x_2)
        &&= \mathds{1}_{\mathcal{B}_1}(x), \\
    \mathds{1}_{\mathcal{B}_{2, 1}}(x)
        &= \mathds{1}_{\mathcal{I}_1^2}(x_1) \mathds{1}_{\mathcal{I}_2^1}(x_2)
        &&= \mathds{1}_{\mathcal{B}_2}(x), \\
    \vdots & && \vdots \\
    \mathds{1}_{\mathcal{B}_{2, 3}}(x)
        &= \mathds{1}_{\mathcal{I}_1^2}(x_1) \mathds{1}_{\mathcal{I}_2^3}(x_2)
        &&= \mathds{1}_{\mathcal{B}_8}(x), \\
    \mathds{1}_{\mathcal{B}_{3, 3}}(x)
        &= \mathds{1}_{\mathcal{I}_1^3}(x_1) \mathds{1}_{\mathcal{I}_2^3}(x_2)
        &&= \mathds{1}_{\mathcal{B}_9}(x). \tag*{\exampleSymbol}
\end{alignat*}
\end{example}

The product formulation of the indicator functions naturally leads to a tensor approximation $ \mathbf{P} $ of the Perron--Frobenius operator $ \mathcal{P} $. Let $ Q_\mu : \R^d \to \R $ be the projection onto the $ \mu $-th component of a vector, i.e.~$ Q_\mu(x) = x_\mu $. Then we define
\begin{equation} \label{eq:Ulam tensor formulation}
    \mathbf{P}[i_1, \dots, i_d, j_1, \dots, j_d] = \frac{1}{n} \sum_{l=1}^n \prod_{\mu=1}^d
        \mathds{1}_{\mathcal{I}_\mu^{j_\mu}} \left( Q_\mu\left(S\left(x_{i_1, \dots, i_d}^{(l)}\right)\right) \right),
\end{equation}
where $ x_{i_1, \dots, i_d}^{(l)} $, $ l = 1, \dots, n $, are the test points generated for box $ \mathcal{B}_{\mathbf{i}} $. Instead of checking to which $ d $-dimensional box the test points are mapped, the $ d $ dimensions are now treated separately.

\begin{proposition} \label{pro:Ulam equivalency}
It holds that
\begin{equation*}
    \mathbf{P} \mathbf{v} = \lambda \mathbf{v}
    \quad \Leftrightarrow \quad
    \hat{P} v = \lambda v.
\end{equation*}
\end{proposition}
\begin{proof}
It suffices to show that $ \mathbf{P}[i_1, \dots, i_d, j_1, \dots, j_d] = \hat{P}_{\hat{i}\hat{j}} $ and that $ (\mathbf{P} \mathbf{v})[i_1, \dots, i_d] = (\hat{P} v)_{\hat{i}} $. The entries of $ \mathbf{P} $ and $ \hat{P} $ are identical since with \eqref{eq:Ulam p_ij} and \eqref{eq:indicator product}
\begin{align*}
    \hat{P}_{\hat{i}\hat{j}}
        &= \frac{1}{n} \sum_{l=1}^{n}
            \mathds{1}_{\mathcal{B}_{\hat{j}}}\left(S\left(x_{\hat{i}}^{(l)}\right)\right) \\
        &= \frac{1}{n} \sum_{l=1}^{n} \prod_{\mu=1}^d
            \mathds{1}_{\mathcal{I}_\mu^{j_\mu}}\left(Q_\mu\left(S\left(x_{i_d, \dots, i_d}^{(l)}\right)\right)\right)
         = \mathbf{P}[i_1, \dots, i_d, j_1, \dots, j_d].
\end{align*}
For $ d = 1 $, the multi-index $ \mathbf{i} $ is mapped to $ \hat{i} = i_1 $ and $ \mathbf{j} $ to $ \hat{j} = j_1 $ by \eqref{eq:Index mapping}. Furthermore, $ \hat{k} = k_1 $ and
\begin{equation*}
    (\mathbf{P} \mathbf{v})[i_1]
        = \sum_{j_1=1}^{k_1} \mathbf{P}[i_1, j_1] \mathbf{v}[j_1]
        = \sum_{\hat{j}=1}^{\hat{k}} \hat{P}_{\hat{i} \hat{j}} v_{\hat{j}}
        = (\hat{P} v)_{\hat{i}}.
\end{equation*}
Then, we obtain by induction
\begin{equation*}
    \begin{split}
    (\mathbf{P} \mathbf{v})[i_1, \dots, i_{d+1}]
        &= \sum_{j_1=1}^{k_1} \dots \sum_{j_{d+1}=1}^{k_{d+1}}
            \mathbf{P}[i_1, \dots, i_{d+1}, j_1, \dots, j_{d+1}]
            \mathbf{v}[j_1, \dots, j_{d+1}] \\
        &= \sum_{j_{d+1}=1}^{k_{d+1}} \left( \sum_{j_1=1}^{k_1} \dots \sum_{j_d=1}^{k_d}
            \mathbf{P}[i_1, \dots, i_d, i_{d+1}, j_1, \dots, j_d, j_{d+1}]
            \mathbf{v}[j_1, \dots, j_d, j_{d+1}] \right) \\
        &= \sum_{j_{d+1}=1}^{k_{d+1}} \left( \sum_{j_1=1}^{k_1} \dots \sum_{j_d=1}^{k_d}
            \mathbf{P}^{(i_{d+1},j_{d+1})}[i_1, \dots, i_d, j_1, \dots, j_d]
            \mathbf{v}^{(j_{d+1})}[j_1, \dots, j_d] \right) \\
        &= \begin{bmatrix}
               \hat{P}^{(i_{d+1},1)} & \dots & \hat{P}^{(i_{d+1},k_{d+1})}
           \end{bmatrix}
           \begin{bmatrix}
               v^{(1)} \\
               \vdots \\[1ex]
               v^{(k_{d+1})}
           \end{bmatrix}
           \begin{matrix}
               \left. \mkern-6mu \right\} \in \R^{k_1 \cdots k_d} \\
               \vdots \\[1ex]
               \left. \mkern-6mu \right\} \in \R^{k_1 \cdots k_d}
           \end{matrix} \\
        &= (\hat{P} v)_{\hat{i}}.
    \end{split}
\end{equation*}
Here, the matrices and vectors with the superscripts $ (i_{d+1},j_{d+1}) $ and $ (j_{d+1}) $, respectively, are obtained by fixing the corresponding indices, that is, these matrices and vectors are lower-dimensional slices of the corresponding higher-dimensional objects. Since the entries of $ v^{(j_{d+1})} $ are indexed by $ (j_1, \dots, j_d) \mapsto \hat{j} = 1 + \sum_{\mu = 1}^d \left( \prod_{\nu=1}^{\mu-1} k_\nu \right) (j_\mu - 1) $, the entries of the larger vector $ v $ -- obtained by stacking the vectors $ v^{(j_{d+1})} \in \R^{k_1 \cdots k_d} $ -- are indexed by
\begin{equation*}
    (j_1, \dots, j_{d+1}) \mapsto \hat{j} = 1 + \sum_{\mu = 1}^d \left( \prod_{\nu=1}^{\mu-1} k_\nu \right) (j_\mu - 1) + (k_1 \cdots k_d)(j_{d+1} - 1) = 1 + \sum_{\mu = 1}^{d+1} \left( \prod_{\nu=1}^{\mu-1} k_\nu \right) (j_\mu - 1).   \qedhere
\end{equation*}
\end{proof}

Analogously, left eigenvectors of $ \mathbf{P} $ correspond to left eigenvectors of $ \hat{P} $. The implementation of the tensor-based formulation is straightforward since only index computations for intervals are required, a numbering of the $ d $-dimensional boxes is not needed anymore. Let $ \mathbb{T} $ be the set of all test points and $ \ind : \R^d \to \mathbb{N}^d $ the functions that returns the corresponding multi-index $ \mathbf{i} $ for a point $ x \in \R^d $ so that $ x_\mu \in I_\mu^{i_\mu} $, $ \mu = 1, \dots, d $. Then, Ulam's method can simply be expressed as:
\begin{algorithmic}
    \For{\textbf{each} test point $ x \in \mathbb{T} $}
        \State $ y = S(x) $
        \State $ \mathbf{i} = \ind(x) $
        \State $ \mathbf{j} = \ind(y) $
        \State $ \mathbf{P}[i_1, \dots, i_d, j_1, \dots, j_d] \leftarrow \mathbf{P}[i_1, \dots, i_d, j_1, \dots, j_d] + \frac{1}{n} $
    \EndFor
\end{algorithmic}
In the standard formulation, the rows of the matrix $ \hat{P} $ sum up to one. Correspondingly, the sum of all entries of each subtensor of $ \mathbf{P} $ with fixed multi-index $ \mathbf{i} $ is one, i.e.
\begin{equation*}
    \sum_{j_1=1}^{k_1} \dots \sum_{j_d=1}^{k_d} \mathbf{P}[i_1, \dots, i_d, j_1, \dots, j_d] = 1.
\end{equation*}

\begin{example} \label{ex:TensorUlamRewritten}
Let us consider Example~\ref{ex:TensorUlam} again. We select $ n $ random test points $ x_{i_1, i_2}^{(l)} $ for each box $ \mathcal{B}_{i_1, i_2} $, $ i_1, i_2 = 1, \dots, 3 $ and $ l = 1, \dots, n $. This leads to a new approximation $ \mathbf{P} \in \R^{3 \times 3 \times 3 \times 3} $. Written in Matlab notation, we would obtain
\begingroup
\allowdisplaybreaks
\begin{alignat*}{2}
    \mathbf{P}[:,:,1,1] =
    \begin{bmatrix}
        0.68 & 0.31 & 0.17 \\
        0.36 & 0.25 & 0.05 \\
        0.07 & 0.06 & 0.01
    \end{bmatrix}, \quad &
    \mathbf{P}[:,:,1,2] =
    \begin{bmatrix}
        0.04 & 0.09 & 0.09 \\
        0.03 & 0.06 & 0.08 \\
        0.02 & 0.01 & 0.02
    \end{bmatrix}, \quad &&
    \mathbf{P}[:,:,1,3] =
    \begin{bmatrix}
        0.09 & 0.39 & 0.60 \\
        0.05 & 0.17 & 0.29 \\
        0.03 & 0.07 & 0.05
    \end{bmatrix}, \\
    \mathbf{P}[:,:,2,1] =
    \begin{bmatrix}
        0.09 & 0.07 & 0.01 \\
        0.06 & 0.05 &    0 \\
        0.12 & 0.06 &    0
    \end{bmatrix}, \quad &
    \mathbf{P}[:,:,2,2] =
    \begin{bmatrix}
        0.02 & 0.01 & 0.01 \\
            0 & 0.01 &    0 \\
            0 &    0 & 0.02
    \end{bmatrix}, \quad &&
    \mathbf{P}[:,:,2,3] =
    \begin{bmatrix}
        0.01 & 0.05 & 0.05 \\
        0.03 & 0.05 & 0.13 \\
        0 & 0.03 & 0.09
    \end{bmatrix}, \\
    \mathbf{P}[:,:,3,1] =
    \begin{bmatrix}
        0.07 & 0.02 & 0.01 \\
        0.40 & 0.18 & 0.06 \\
        0.64 & 0.37 & 0.03
    \end{bmatrix}, \quad &
    \mathbf{P}[:,:,3,2] =
    \begin{bmatrix}
            0 & 0.02 & 0.01 \\
        0.02 & 0.02 & 0.05 \\
        0.07 & 0.04 & 0.11
    \end{bmatrix}, \quad &&
    \mathbf{P}[:,:,3,3] =
    \begin{bmatrix}
            0 & 0.04 & 0.05 \\
        0.05 & 0.21 & 0.34 \\
        0.05 & 0.36 & 0.67
    \end{bmatrix}.
\end{alignat*}
\endgroup
Note that each matrix $ \mathbf{P}[:, :, j_1, j_2] $ corresponds to a column of matrix $ \hat{P} $ in Figure~\ref{fig:Grid example}. For the resulting eigenvalue problem, we obtain -- using a simple power iteration, see Section~\ref{sec:Eigenvalue problems} -- the left eigenvector $ \mathbf{v}_1 $ corresponding to the largest eigenvalue $ \lambda_1 = 1 $
\begin{equation*}
    \mathbf{v}_1 =
    \begin{bmatrix}
        0.6503 & 0.1393 & 0.4501 \\
        0.1046 & 0.0261 & 0.0901 \\
        0.4355 & 0.0864 & 0.3719
    \end{bmatrix},
\end{equation*}
which is a good approximation of the largest left eigenvector $ v_1 $ of the matrix $ \hat{P} $ given by
\begin{equation*}
    v_1 =
    \begin{bmatrix}
        0.6503 & 0.1393 & 0.4501 & 0.1046 & 0.0261 & 0.0901 & 0.4355 & 0.0864 & 0.3719
    \end{bmatrix}. \tag*{\exampleSymbol}
\end{equation*}
\end{example}

Instead of working with the full format, the matrix $ \mathbf{P} $ can also be expressed directly using the canonical tensor format. Assume that a test point $ x $ is mapped from box $ \mathcal{B}_\mathbf{i} $ to box $ \mathcal{B}_\mathbf{j} $, where $ \mathbf{i} = (i_1, \dots, i_d) $ and $ \mathbf{j} = (j_1, \dots, j_d) $ are again multi-indices. Now let $ e_\mu^{i_\mu} \in \R^{k_\mu} $ be the $ i_\mu $-th unit vector of size $ k_\mu $ and let
\begin{equation*}
    \mathbf{e}^{\mathbf{i}} = \bigotimes_{\mu=1}^d e_\mu^{i_\mu}.
\end{equation*}
Then the elementary tensor $ \mathbf{e}^{\mathbf{i}} \otimes \mathbf{e}^{\mathbf{j}} \in \R^{k_1 \times \dots \times k_d \times k_1 \times \dots \times k_d} $ describes the mapping of this point. Furthermore, let $ \ind $ be again the function that returns the multi-index of the box that contains the point $ x $ and $ \mathbb{T} $ the set of all test points, then the matrix $ \mathbf{P} $ can be represented as a sum of elementary tensors of this form, i.e.
\begin{equation*}
    \mathbf{P} = \frac{1}{n} \sum_{x \in \mathbb{T}}
        \mathbf{e}^{\ind(x)} \otimes \mathbf{e}^{\ind(S(x))}.
\end{equation*}
That is, the number of elementary tensors is $ \hat{k} n $, i.e.~the number of boxes multiplied by the number of test points per box, and thus potentially too large to store (unless sparse tensor formats are used). However, we will store only a low-rank approximation to reduce the required storage space. Note that this elementary tensor representation can also be easily converted into the TT format for numerical computations using the TT toolbox.

The question now is whether the tensor representation offers advantages over the standard formulation of Ulam's method. The goal is to approximate the eigenfunctions of the Perron--Frobenius operator or Koopman operator using low-rank tensors, reducing the computational cost as well as the memory consumption. Before we present numerical results, let us also rewrite EDMD in tensor form.

\subsection{Reformulation of EDMD}

Instead of writing $ \Psi $ as a vector of functions $ \Psi = [\psi_1, \psi_2, \dots, \psi_{\hat{k}}]^T $, we now write $ \Psi $ as a tensor of functions. We start by selecting basis functions for each dimension separately. Let
\begin{equation*}
    \mathcal{D}_\mu = \{ \psi_\mu^1, \dots, \psi_\mu^{k_\mu} \}
\end{equation*}
be the set of basis functions for dimension $ \mu $, $ \mu = 1, \dots, d $. Here, each $ \psi_\mu^{i_\mu} : \R \to \R $ depends only on $ x_\mu $. Then our tensor basis for EDMD contains all functions of the form
\begin{equation} \label{eq:EDMD product basis}
    \boldsymbol{\psi}^\mathbf{i}(x) = \prod_{\mu = 1}^d \psi_\mu^{i_\mu}(x_\mu),
\end{equation}
where $ \mathbf{i} = (i_1, \dots, i_d) $ is a multi-index. Thus,
\begin{equation*}
    \mathcal{D} =
        \left\{ \prod_{\mu = 1}^d \psi_\mu^{i_\mu}, \psi_\mu^{i_\mu} \in \mathcal{D}_\mu \right\}.
\end{equation*}
That is, we have again $ \hat{k} = \prod_{\mu=1}^d k_\mu $ basis functions and $ \boldsymbol{\Psi} : \R^d \to \R^{k_1 \times \dots \times k_d} $, with
\begin{equation*}
    \boldsymbol{\Psi}[i_1, \dots, i_d](x) = \boldsymbol{\psi}^\mathbf{i}(x).
\end{equation*}

\begin{example}
Let us begin with a simple example: Assume we have a two-dimensional domain $ \mathcal{X} \subset \R^2 $ and we want to use monomials of order up to three $ \{ 1, x_\mu, x_\mu^2, x_\mu^3 \} $ in $ x_1 $ and $ x_2 $ direction to approximate the eigenfunctions of the Koopman operator. Written in tensor form, we obtain
\begin{equation*}
    \boldsymbol{\Psi}(x) =
    \begin{bmatrix}
        1     &       x_2 &       x_2^2 &       x_2^3 \\
        x_1   & x_1   x_2 & x_1   x_2^2 & x_1   x_2^3 \\
        x_1^2 & x_1^2 x_2 & x_1^2 x_2^2 & x_1^2 x_2^3 \\
        x_1^3 & x_1^3 x_2 & x_1^3 x_2^2 & x_1^3 x_2^3
    \end{bmatrix}.
\end{equation*}
That is, $ \boldsymbol{\Psi}[i_1, i_2](x) = x_1^{i_1-1} x_2^{i_2-1} $. Analogously, for a $ d $-dimensional domain, we would obtain $ \boldsymbol{\Psi}(x) \in \R^{k_1 \times \cdots \times k_d} $ with
\begin{equation*}
    \boldsymbol{\Psi}[i_1, \dots, i_d](x) = x_1^{i_1-1} \, \cdots \, x_d^{i_d-1}. \tag*{\exampleSymbol}
\end{equation*} 
\end{example}

Such a tensor basis is often used for high-dimensional problems, see also \cite{FJK09}. Typical basis functions are monomials, Hermite polynomials, or trigonometric functions. In the standard formulation, all basis functions are enumerated and rewritten in vector form \eqref{eq:Psi}. The difference here is that the tensor form will be preserved. We will use again \eqref{eq:Index mapping} as a mapping from multi-index to single index when required.

Now $ \mathbf{A}, \mathbf{G} \in \R^{k_1 \times \dots \times k_d \times k_1 \times \dots \times k_d} $ can be constructed as follows:
\begin{equation*}
    \begin{split}
        \mathbf{A}[i_1, \dots, i_d, j_1, \dots, j_d]
            &= \innerprod{\mathcal{K} \mathbf{\Psi}[i_1, \dots, i_d]}{\mathbf{\Psi}[j_1, \dots, j_d]}, \\
        \mathbf{G}[i_1, \dots, i_d, j_1, \dots, j_d]
            &= \innerprod{\mathbf{\Psi}[i_1, \dots, i_d]}{\mathbf{\Psi}[j_1, \dots, j_d]}.
    \end{split}
\end{equation*}
The entries are again -- as in the standard EDMD formulation -- approximated using a collocation approach. EDMD computes the entries as shown in~\eqref{eq:A and G entries}, for the new tensor-based formulation this results in
\begin{equation*}
    \begin{split}
        \mathbf{A}[i_1, \dots, i_d, j_1, \dots, j_d] &= \frac{1}{m} \sum_{l=1}^m
            \boldsymbol{\Psi}[i_1, \dots, i_d](y_l) \boldsymbol{\Psi}[j_1, \dots, j_d](x_l), \\
        \mathbf{G}[i_1, \dots, i_d, j_1, \dots, j_d] &= \frac{1}{m} \sum_{l=1}^m
            \boldsymbol{\Psi}[i_1, \dots, i_d](x_l) \boldsymbol{\Psi}[j_1, \dots, j_d](x_l),
    \end{split}
\end{equation*}
or in short form, using the outer product,
\begin{equation} \label{eq:tensor A and G}
    \begin{split}
        \mathbf{A} &= \frac{1}{m} \sum_{l=1}^m
            \boldsymbol{\Psi}(y_l) \otimes \boldsymbol{\Psi}(x_l), \\
        \mathbf{G} &= \frac{1}{m} \sum_{l=1}^m
            \boldsymbol{\Psi}(x_l) \otimes \boldsymbol{\Psi}(x_l),
    \end{split}
\end{equation}
which in turn results in a generalized eigenvalue problem of the form
\begin{equation*}
    \boldsymbol{\xi} \mathbf{A} = \lambda \boldsymbol{\xi} \mathbf{G}.
\end{equation*}
Note that the eigenvalue problem is the same as \eqref{eq:EDMD generalized eig K} in the standard case. For the sake of simplicity, we are omitting the index $ \mathbf{i} $ here.

\begin{proposition}
Provided that the basis functions can be written in tensor product form \eqref{eq:EDMD product basis},
\begin{equation*}
    \boldsymbol{\xi} \mathbf{A} = \lambda \boldsymbol{\xi} \mathbf{G}
    \quad \Leftrightarrow \quad
    \xi \hat{A} = \lambda \xi \hat{G}.
\end{equation*}
\end{proposition}
\begin{proof}
We just show that the entries of $ \mathbf{A} $ and $ \hat{A} $ as well as the entries of $ \mathbf{G} $ and $ \hat{G} $ are identical, the rest -- the equivalency of the matrix-vector and tensor products -- follows from Proposition~\ref{pro:Ulam equivalency}. Assuming that the basis can be written in the product form, we obtain from~\eqref{eq:A and G entries}
\begin{equation*}
    \begin{split}
        \hat{a}_{\hat{i}\hat{j}}
            &= \frac{1}{m} \sum_{l=1}^m \psi_{\hat{i}}(y_l) \psi_{\hat{j}}(x_l) \\
            &= \frac{1}{m} \sum_{l=1}^m \prod_{\mu = 1}^d \psi_\mu^{i_\mu}(Q_\mu(y_l))
                                        \prod_{\mu = 1}^d \psi_\mu^{j_\mu}(Q_\mu(x_l)) \\
            &= \frac{1}{m} \sum_{l=1}^m \boldsymbol{\Psi}[i_1, \dots, i_d](y_l)
                                        \boldsymbol{\Psi}[j_1, \dots, j_d](x_l) \\
            &= \mathbf{A}[i_1, \dots, i_d, j_1, \dots, j_d]
    \end{split}
\end{equation*}
and analogously $ \hat{g}_{\hat{i}\hat{j}} = \mathbf{G}[i_1, \dots, i_d, j_1, \dots, j_d] $. Here, $ Q_\mu $ is again the projection onto the $ \mu $-th component of a vector, cf.~\eqref{eq:Ulam tensor formulation}.
\end{proof}

Instead of storing the dense matrices $ \mathbf{A} $ and $ \mathbf{G} $, we can again directly represent these matrices using the canonical tensor format. The basis was chosen in such a way that $ \boldsymbol{\Psi}(x) $ can be written as
\begin{equation*}
    \boldsymbol{\Psi}(x) = \bigotimes_{\mu=1}^d \tilde{\psi}_\mu(x_\mu),
\end{equation*}
where $ \tilde{\psi}_\mu = [\psi_\mu^1, \dots, \psi_\mu^{k_\mu}]^T \in \R^{k_\mu} $. 
With \eqref{eq:tensor A and G} it follows that $ \mathbf{A} $ and $ \mathbf{G} $ can be written as sums of $ m $ elementary tensors. As before, we are not storing the full-rank tensor, but only low-rank approximations.

The eigentensors $ \boldsymbol{\xi} \in \R^{k_1 \times \dots \times k_d} $ of the generalized eigenvalue problem can then be used to approximate the eigenfunctions of the Perron--Frobenius operator or Koopman operator: Let $ \boldsymbol{\xi} $ be a left eigentensor, then
\begin{equation*}
    \varphi(x) = \innerprod{\boldsymbol{\xi}}{\boldsymbol{\Psi}(x)}
\end{equation*}
approximates an eigenfunction of the Koopman operator. Analogously, if $ \boldsymbol{\xi} $ is a right eigentensor of the generalized eigenvalue problem -- observe that $ \mathbf{G}[i_1, \dots, i_d, j_1, \dots, j_d] = \mathbf{G}[j_1, \dots, j_d, i_1, \dots, i_d] $ --, then $ \varphi(x) $ is an approximation of the corresponding eigenfunction of the Perron--Frobenius operator, see also~\cite{KKS15}.

To compute the dominant eigenfunctions of the Koopman operator or Perron--Frobenius operator, we will use simple power iteration schemes outlined in the next section. General purpose eigenvalue solvers for nonsymmetric generalized eigenvalue problems are, to our knowledge, not part of the tensor libraries yet. Solvers for symmetric (non-generalized) eigenvalue problems already exist and are part of the TT toolbox~\cite{Ose14}.

\section{Eigenvalue problems}
\label{sec:Eigenvalue problems}

In Section~\ref{sec:Perron--Frobenius and Koopman operator approximation} and Section~\ref{sec:Tensor-based approximation}, we have shown that in order to compute the eigenfunctions of the Perron--Frobenius operator and Koopman operator, respectively, using either Ulam's method or EDMD, we need to solve standard eigenvalue problems or generalized eigenvalue problems. For the reformulated version of these methods, we have to develop the required numerical algorithms to solve the resulting tensor-based eigenvalue problems. At the time of writing, we are not aware of any tensor toolbox containing numerical methods for nonsymmetric generalized eigenvalue problems. Methods for the computation of eigenvectors of symmetric positive definite matrices in the TT format have been proposed in \cite{HRS12}, where the eigenvalue problem is rewritten as a (Rayleigh quotient based) minimization problem which is then solved using the \emph{Alternating Linear Scheme} (ALS). In practice, these methods have recently also been successfully used for nonsymmetric problems, although convergence has not been shown yet~\cite{GMS16}.

Suitable methods for eigenvalue problems can be subdivided into two main categories as explained in~\cite{GKT13} (see also references therein, e.g.~\cite{BM05}): The first category of methods is based on combining classical iterative algorithms with low-rank truncation after each step, the second is based on a reformulation as an optimization problem, where admissible solutions are constrained to the set of low-rank tensors. In this section, we will describe a generalization of simple power iteration methods -- belonging to the first category -- to tensor-based eigenvalue problems. For a detailed description of general power iteration methods, we refer to~\cite{GVL13}. Power iteration and inverse power iteration for tensors have also been proposed in~\cite{BM05}. The main difference between the standard algorithms and the tensor-based counterpart is that for the latter truncation is used to keep the ranks of the tensors low. It is important that the iteration moves from the initial state to the final state without creating intermediate solutions with an excessive rank~\cite{BM05}.

\subsection{Power iteration methods for standard eigenvalue problems}

In what follows, let $ \mathcal{T} $ denote the truncation of a tensor. Then instead of a classical iteration scheme of the form $ x_{k+1} = F(x_k) $, we simply obtain $ \mathbf{x}_{k+1} = \mathcal{T}(F(\mathbf{x}_k)) $. For an eigenvalue problem of the form $ \mathbf{A} \mathbf{v} = \lambda \mathbf{v} $, given an initial guess $ \mathbf{v}_0 $ for the dominant eigenvector, the power iteration algorithm computes:
\begin{algorithmic}
    \For{$ k = 1, 2, \dots $}
        \State $ \mathbf{w}^{(k)} = \mathcal{T}\left(\mathbf{A} \mathbf{v}^{(k-1)}\right) $
        \State $ \mathbf{v}^{(k)} = \mathbf{w}^{(k)} / \norm{\mathbf{w}^{(k)}} $
        \State $ \lambda^{(k)} = \innerprod{\mathbf{v}^{(k)}}{\mathcal{T}(\mathbf{A} \mathbf{v}^{(k)})} $
    \EndFor
\end{algorithmic}
The iteration converges to an eigenvector associated with the largest eigenvalue $ \lambda_1 $ of the truncated operator $ \mathbf{A} $ if the eigenvalue is simple and the initial guess $ \mathbf{v}^{(0)} $ has a component in the direction of the corresponding dominant eigenvector $ \mathbf{v}_1 $ \cite{GVL13}. Even if the initial guess does not have a component in the direction of $ \mathbf{v}_1 $, rounding errors typically ensure that this direction will be picked up during the iteration. The rate of convergence depends on the ratio between the second-largest and largest eigenvalue $ \lambda_2 / \lambda_1 $. The main advantage of this method is that it requires only matrix-vector multiplications and can thus easily be used for tensor eigenvalue problems.

A modification of this algorithm to compute eigenvectors corresponding to any eigenvalue is the inverse power iteration with shift, which -- assuming that $ \mathbf{A} - \theta \mathbf{I} $ is nonsingular -- can be written in the form:
\begin{algorithmic}
    \For{$ k = 1, 2, \dots $}
        \State Solve $ (\mathbf{A} - \theta \mathbf{I}) \mathbf{w}^{(k)} = \mathbf{v}^{(k-1)} $
        \State $ \mathbf{v}^{(k)} = \mathbf{w}^{(k)} / \norm{\mathbf{w}^{(k)}} $
        \State $ \lambda^{(k)} = \innerprod{\mathbf{v}^{(k)}}{\mathbf{A} \mathbf{v}^{(k)}} $
    \EndFor
\end{algorithmic}
The parameter $ \theta $ is called shift and the iteration converges to the eigenvalue closest to $ \theta $. This method is just the standard power iteration applied to the matrix $ (\mathbf{A} - \theta \mathbf{I})^{-1} $. Here, the linear solver computes a low-rank approximation of the solution so that truncation is not required.

\subsection{Power iteration methods for generalized eigenvalue problems}

Given matrices $ \mathbf{A} $ and $ \mathbf{B} $, the generalized eigenvalue problem is given by $ \mathbf{A} \mathbf{v} = \lambda \mathbf{B} \mathbf{v} $. In this case, the power iteration method also requires the solution of a linear system of equations. Thus, we can also directly apply the inverse power iteration, where the resulting systems of linear equations are again solved with ALS:
\begin{algorithmic}
    \For{$ k = 1, 2, \dots $}
        \State Solve $ (\mathbf{A} - \theta \mathbf{B}) \mathbf{w}^{(k)}
                = \mathbf{B} \mathbf{v}^{(k-1)} $
        \State $ \mathbf{v}^{(k)} = \mathbf{w}^{(k)} / \norm{\mathbf{w}^{(k)}} $
        \State $ \lambda^{(k)} = \innerprod{\mathbf{v}^{(k)}}{\mathbf{A} \mathbf{v}^{(k)}}
                               / \innerprod{\mathbf{v}^{(k)}}{\mathbf{B} \mathbf{v}^{(k)}} $
    \EndFor
\end{algorithmic}

In order to keep the ranks of the intermediate solutions low, we also approximate the matrices $ \mathbf{P} $ (Ulam's method) or $ \mathbf{A} $ and $ \mathbf{G} $ (EDMD) by low-rank tensors. That is, the initial matrices are converted to $ \mathbf{\tilde{P}} $, $ \mathbf{\tilde{A}} $, and $ \mathbf{\tilde{G}} $ with a lower rank since the rank of these matrices can initially be very high. We are typically only interested in the general behavior of the eigenfunctions, a highly accurate representation of the eigenfunctions is often not needed.

\section{Examples}
\label{sec:Examples}

All examples presented within this section have been implemented in Matlab using -- for the sake of efficiency~-- \texttt{mex}-functions to integrate the SDEs. All tensor computations were carried out with the \emph{TT toolbox}~\cite{Ose14}. For the eigenvector computations, we used our implementation of the simple power iteration methods described in Section~\ref{sec:Eigenvalue problems}.

\subsection{2-dimensional double well problem}

Let us start with a simple 2-dimensional example, a stochastic differential equation of the form
\begin{equation*}
    \begin{split}
        dx_1 &= -\grad_{x_1} V(x_1, x_2) \, dt + \sigma \, dW_1, \\
        dx_2 &= -\grad_{x_2} V(x_1, x_2) \, dt + \sigma \, dW_2,
    \end{split}
\end{equation*}
where $ W_1 $ and $ W_2 $ are two independent standard Wiener processes. Here, the potential is given by
\begin{equation*}
    V(x_1, x_2) = (x_1^2 - 1)^2 + x_2^2,
\end{equation*}
see Figure~\ref{fig:Double well potential} (cf.~\cite{KKS15}). Furthermore, we set $ \sigma = 0.7 $. Note that in this case the potential can be written as $ V(x_1, x_2) = V_1(x_1) + V_2(x_2) $.
\begin{figure}[htb]
    \centering
    \includegraphics[width=0.4\textwidth]{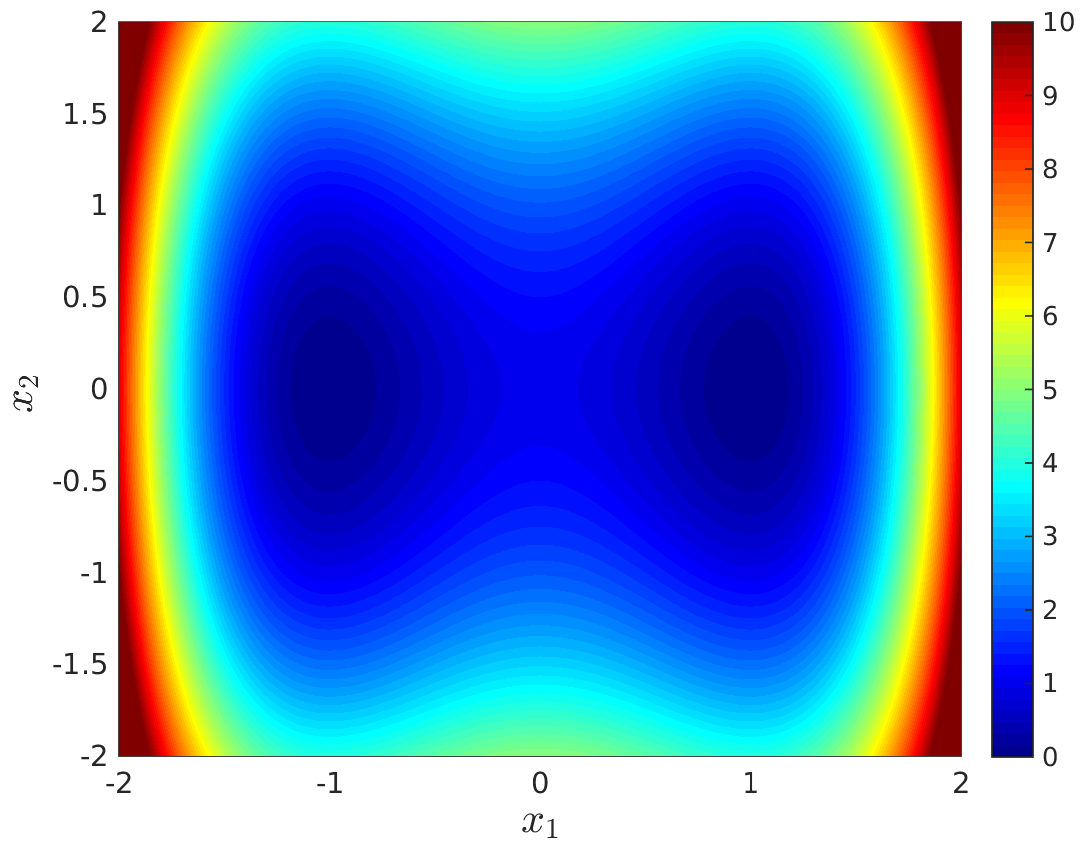}
    \caption{Double-well potential $ V(x_1, x_2) = (x_1^2 - 1)^2 + x_2^2 $.}
    \label{fig:Double well potential}
\end{figure}
In order to analyze the tensor-based methods, we rotate the potential by an angle $ \alpha $ and obtain
\begin{equation*}
    \widetilde{V}(x_1, x_2) = \left((\cos(\alpha) \, x_1 - \sin(\alpha) \, x_2)^2 - 1\right)^2
                             + \left(\sin(\alpha) \, x_1 + \cos(\alpha) \, x_2\right)^2.
\end{equation*}
The two independent Wiener processes $ W_1 $ and $ W_2 $ are rotated accordingly. We would expect that the eigenfunctions of systems with small $ \alpha $ can be accurately approximated by low-rank tensors, whereas systems with a larger value of $ \alpha $ require higher ranks since the dynamics are not aligned with the axes anymore.

\begin{figure}[htb]
    \newcommand*{\factori}{0.2312}
    \newcommand*{\factorii}{0.0353}
    \centering
    \includegraphics[width=\factori\textwidth]{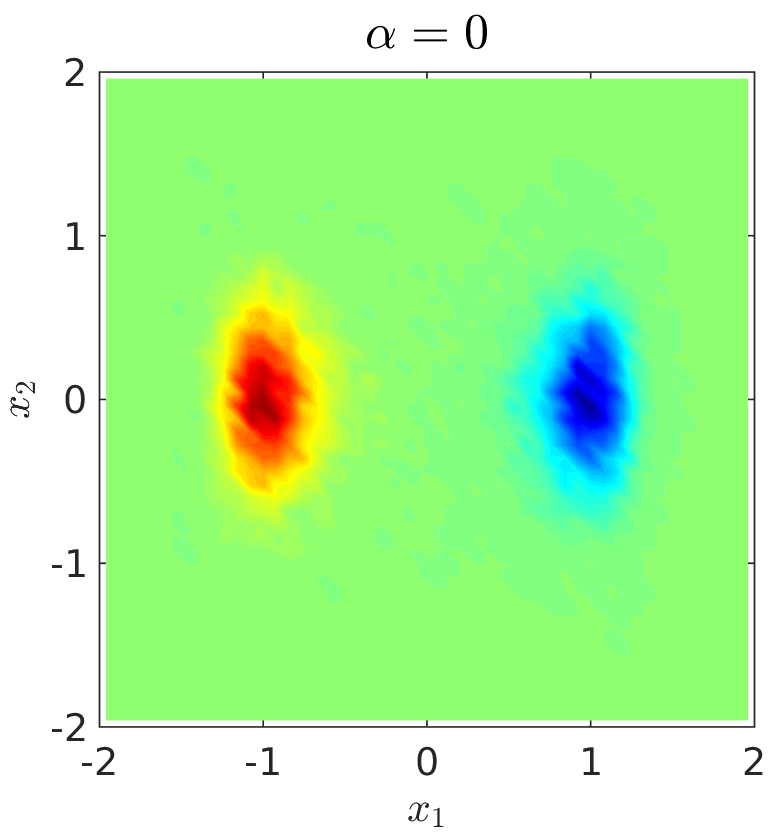}
    \includegraphics[width=\factori\textwidth]{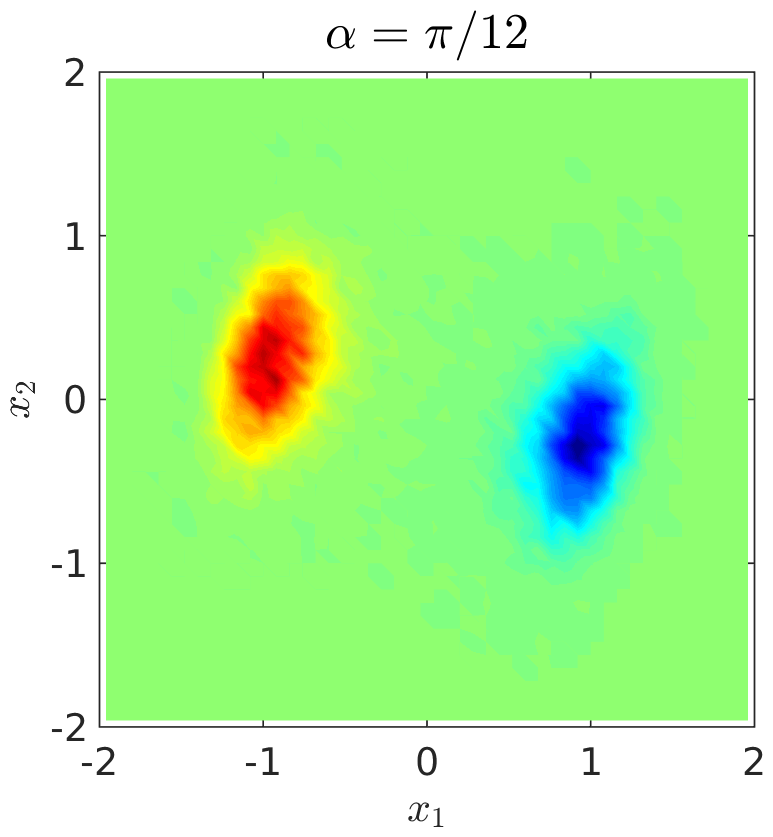}
    \includegraphics[width=\factori\textwidth]{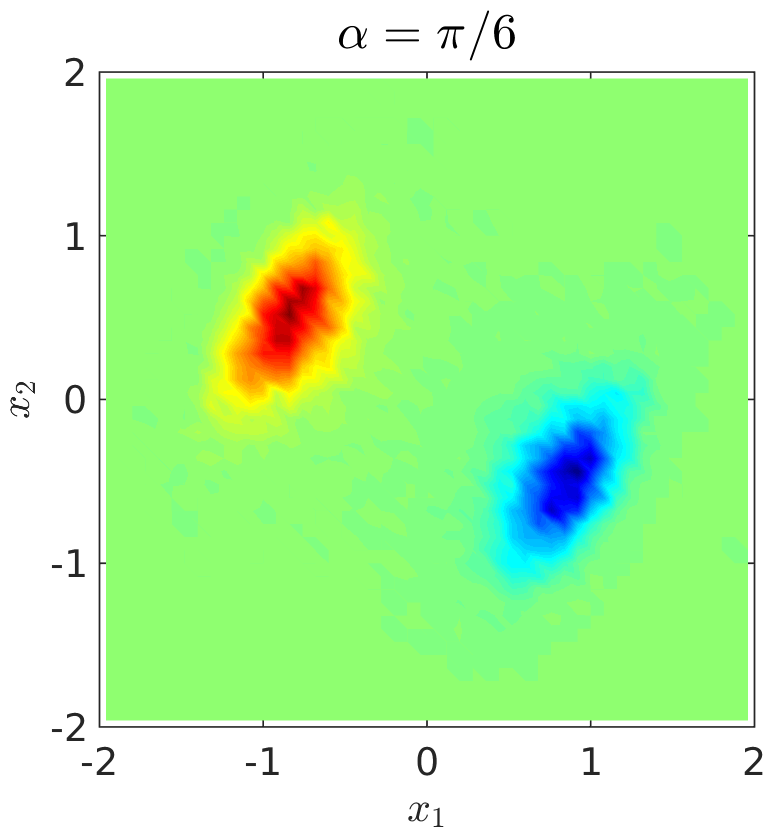}
    \includegraphics[width=\factori\textwidth]{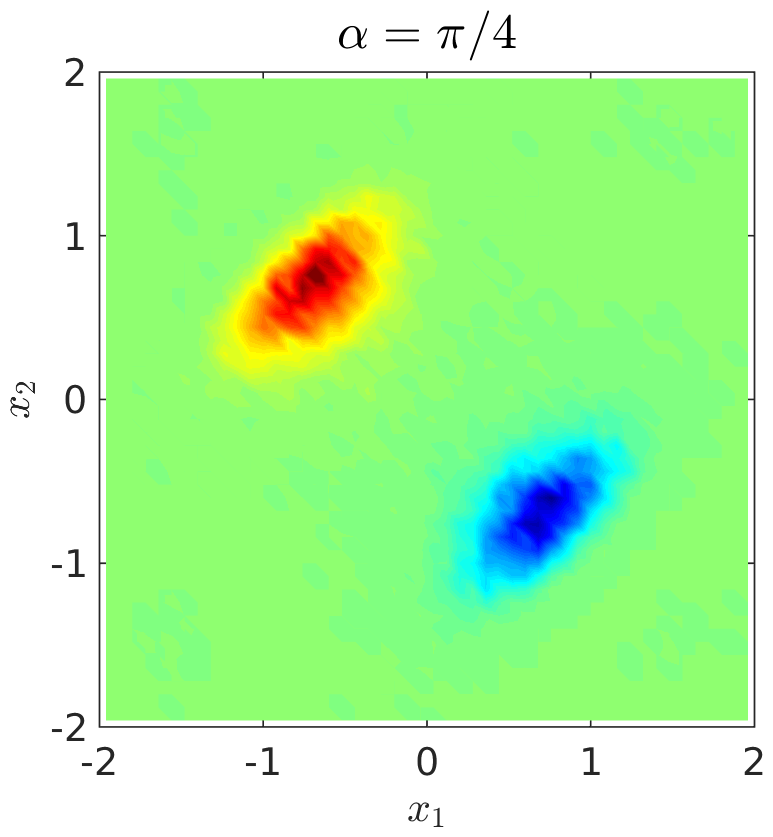}
    \includegraphics[width=\factorii\textwidth]{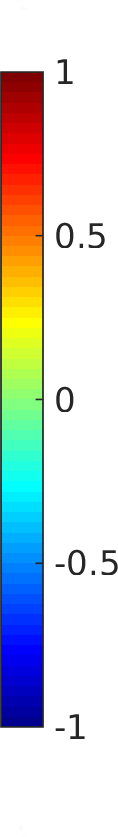}
    \caption{Second eigenfunction of the Perron--Frobenius operator for different values of $ \alpha $.}
    \label{fig:Double well P2}
\end{figure}

The second eigenfunctions of the Perron--Frobenius operator for the systems with potential $ \widetilde{V} $ and different values of $ \alpha $ computed using the tensor-based version of Ulam's method are shown in Figure~\ref{fig:Double well P2}. We chose $ \alpha = 0 $, $ \alpha = \pi/12 $, $ \alpha = \pi/6 $, and $ \alpha = \pi/4 $. The domain $ \mathcal{X} = [-2, 2]^2 $ was subdivided into $ 50 \times 50 $ equally sized boxes. That is, $ \mathbf{P} \in \R^{50 \times 50 \times 50 \times 50} $. For each box, 100 randomly chosen test points were generated. The Euler--Maruyama method with a step size $ h = 10^{-3} $ was used for the numerical integration, where one evaluation of $ S $ corresponds to $ 10,000 $ integration steps, that is, the integration interval is $ [0, 10] $. The shift parameter $ \theta $ of the power iteration method was set to a value slightly smaller than $ 1 $. Figure~\ref{fig:Double well P2r1248} illustrates how the tensor approximation, depending on the rank, successively picks up the information about the shape of the eigenfunction and generates more and more accurate representations. For $ \alpha = \pi/4 $, a tensor of rank $ 1 $ cannot represent the minimum and maximum simultaneously since this would lead to two additional peaks in the lower left and upper right corner. Here, the first pair of singular vectors represents the maximum, the second pair of singular vectors the minimum.
\begin{figure}[htb]
    \newcommand*{\factori}{0.2312}
    \newcommand*{\factorii}{0.0353}
    \centering
    \includegraphics[width=\factori\textwidth]{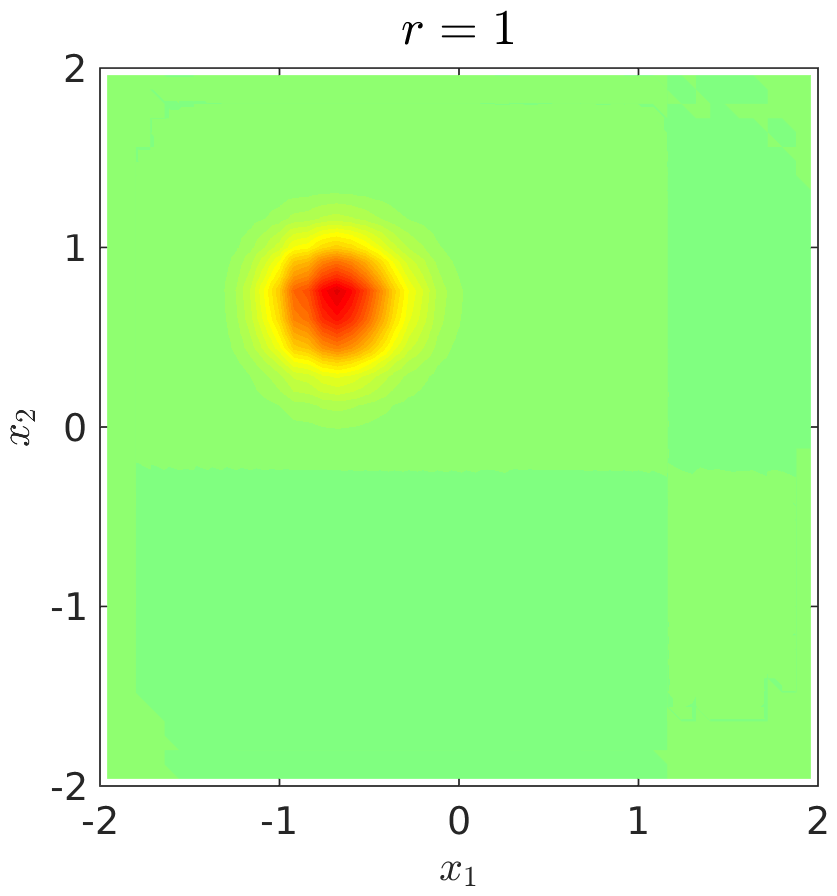}
    \includegraphics[width=\factori\textwidth]{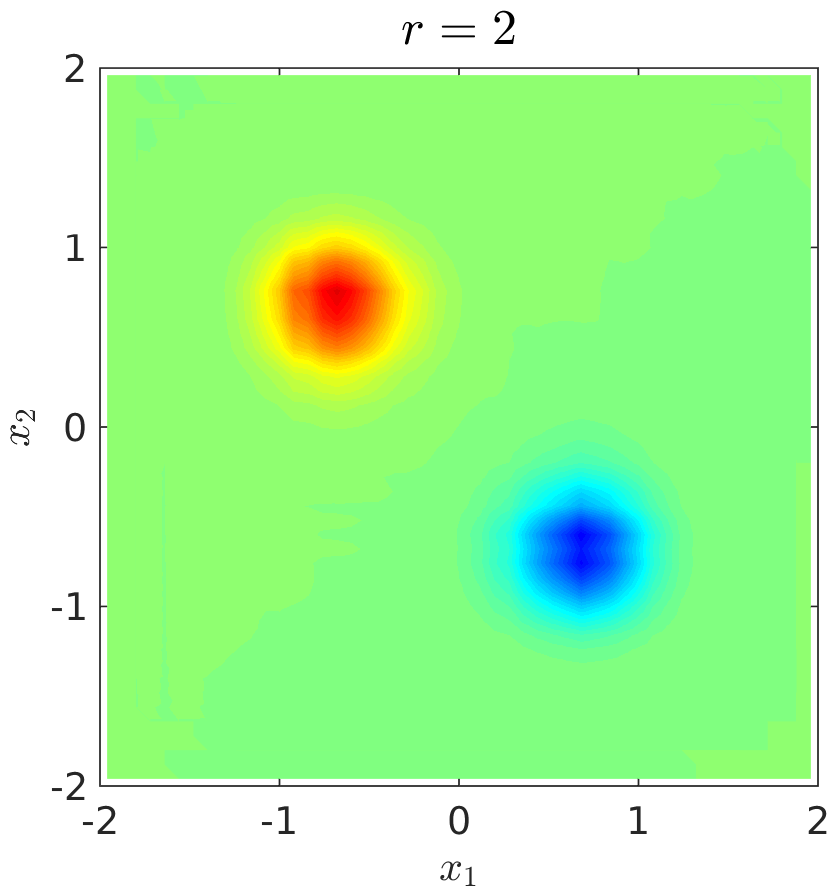}
    \includegraphics[width=\factori\textwidth]{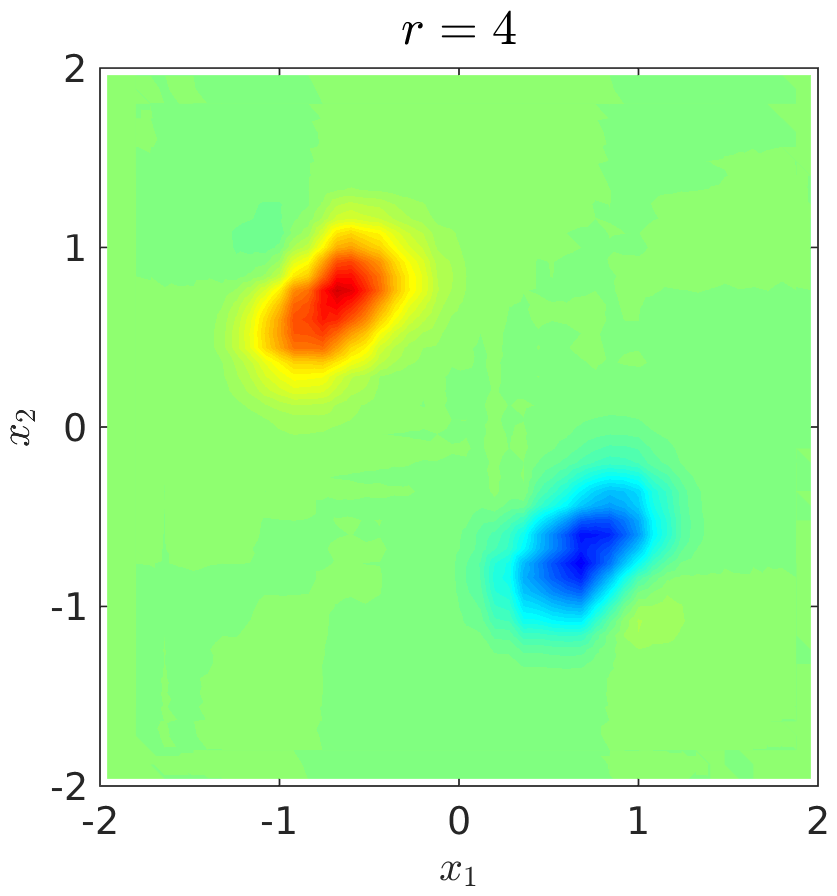}
    \includegraphics[width=\factori\textwidth]{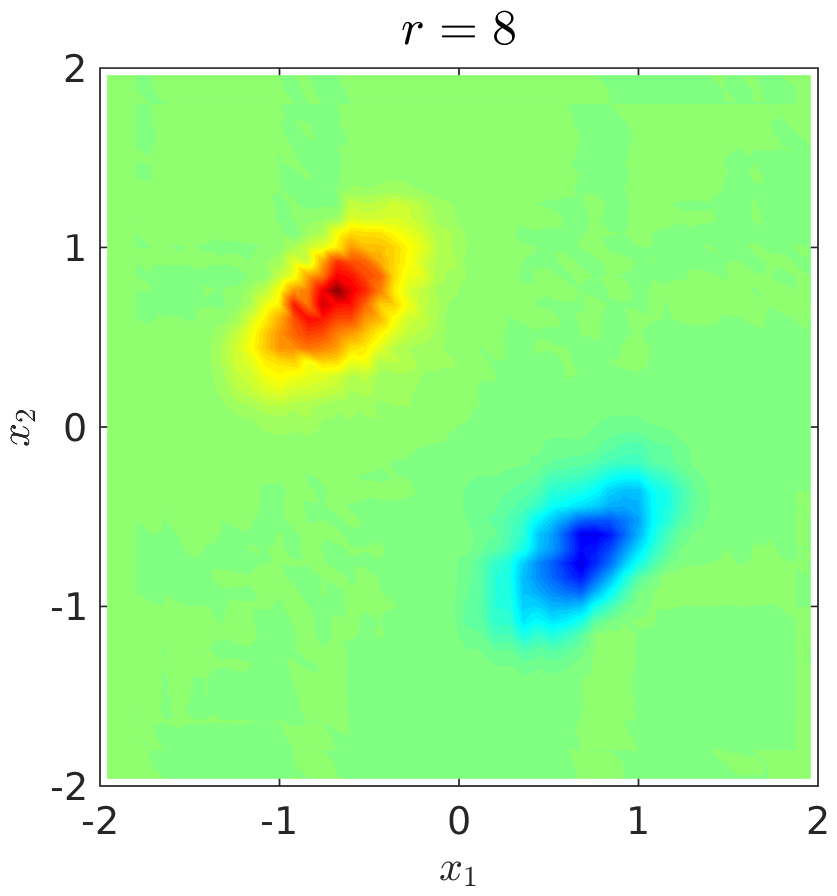}
    \includegraphics[width=\factorii\textwidth]{pics/ColorBar}
    \caption{Tensor approximations of the second eigenfunction of the Perron--Frobenius operator for $ \alpha = \pi/4 $ with increasing rank $ r $.}
    \label{fig:Double well P2r1248}
\end{figure}

Additionally, we computed the eigenfunctions with the standard version of Ulam's method to evaluate the accuracy of the approximation and compared it with the results obtained by using the new tensor-based formulation. Figure~\ref{fig:Double well truncation} shows the influence of the truncation of the operator as well as the influence of the truncation of the resulting eigenfunctions. Here, in order to analyze the accuracy, we also compare the first eigenfunction with the analytically computed invariant density. Since we are computing eigenfunctions of the Perron--Frobenius operator associated with a stochastic differential equation, the results depend strongly on the number of test points chosen for each box. The higher the number of test points per box, the smoother the eigenfunction approximation. Thus, in this case, the smoother low-rank solutions can counterintuitively lead to better approximations of the true eigenfunctions. The high ranks are mainly required to resolve the numerical noise introduced by the coarse approximation of the operator. This can be seen, for example, in Figure~\ref{fig:Double well truncation}b. Decreasing the rank initially reduces the error -- the truncation of the operator results in smoother eigenfunctions -- until the shape of the eigenfunction cannot be described by a low-rank approximation anymore and the error increases. For $ \alpha = 0 $, the $ x_1 $ and $ x_2 $ dynamics are independent and a low-rank approximation is sufficient. Furthermore, the results illustrate that for a fixed-rank approximation, the error is smaller when the system's dynamics are aligned with the axes.
\begin{figure}[htb]
    \centering
    \begin{minipage}[b]{0.49\textwidth}
        \centering
        \subfiguretitle{a)}
        \includegraphics[width=\textwidth]{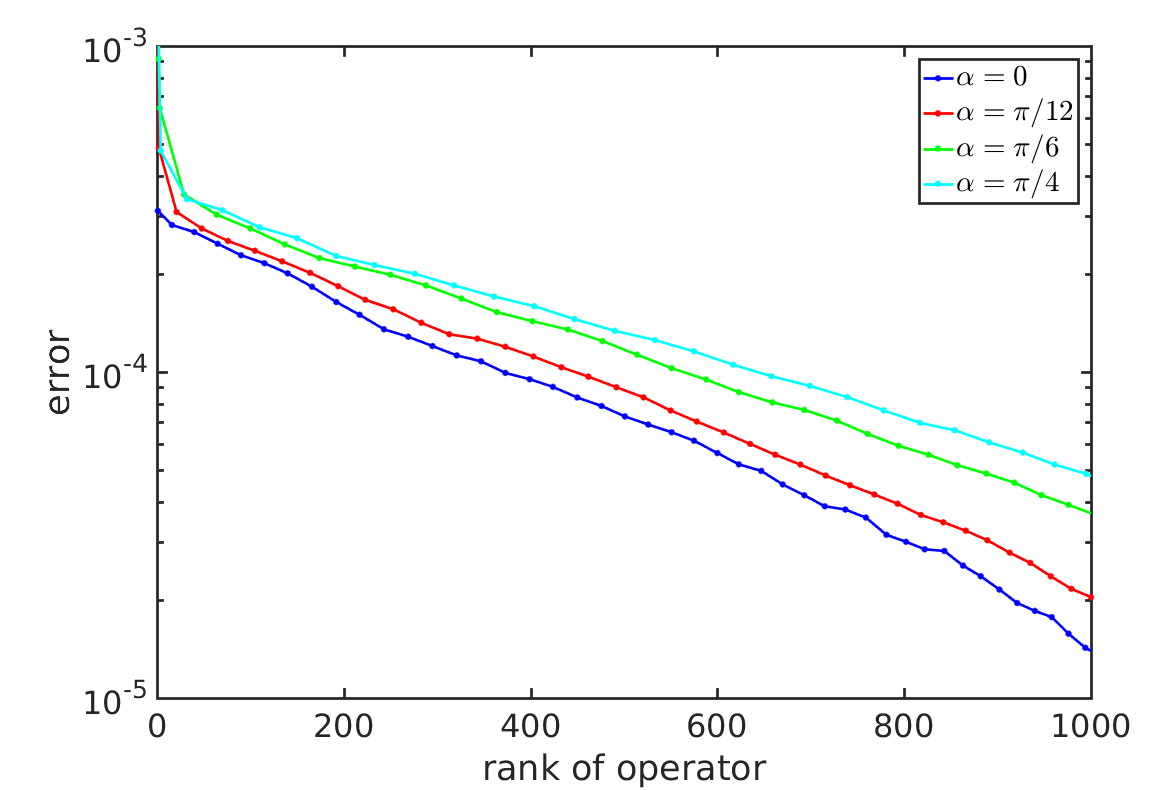}
    \end{minipage}
    \begin{minipage}[b]{0.49\textwidth}
        \centering
        \subfiguretitle{b)}
        \includegraphics[width=\textwidth]{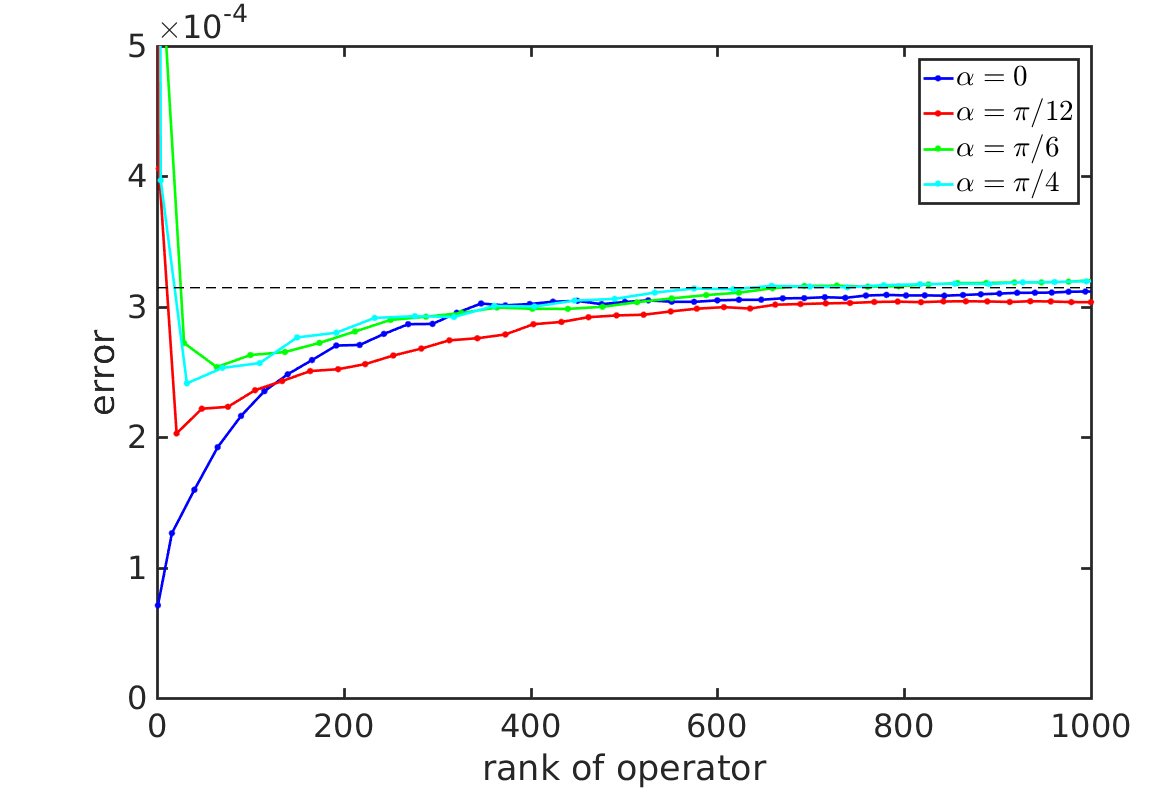}
    \end{minipage}
    \\[1ex]
    \begin{minipage}[b]{0.49\textwidth}
        \centering
        \subfiguretitle{c)}
        \includegraphics[width=\textwidth]{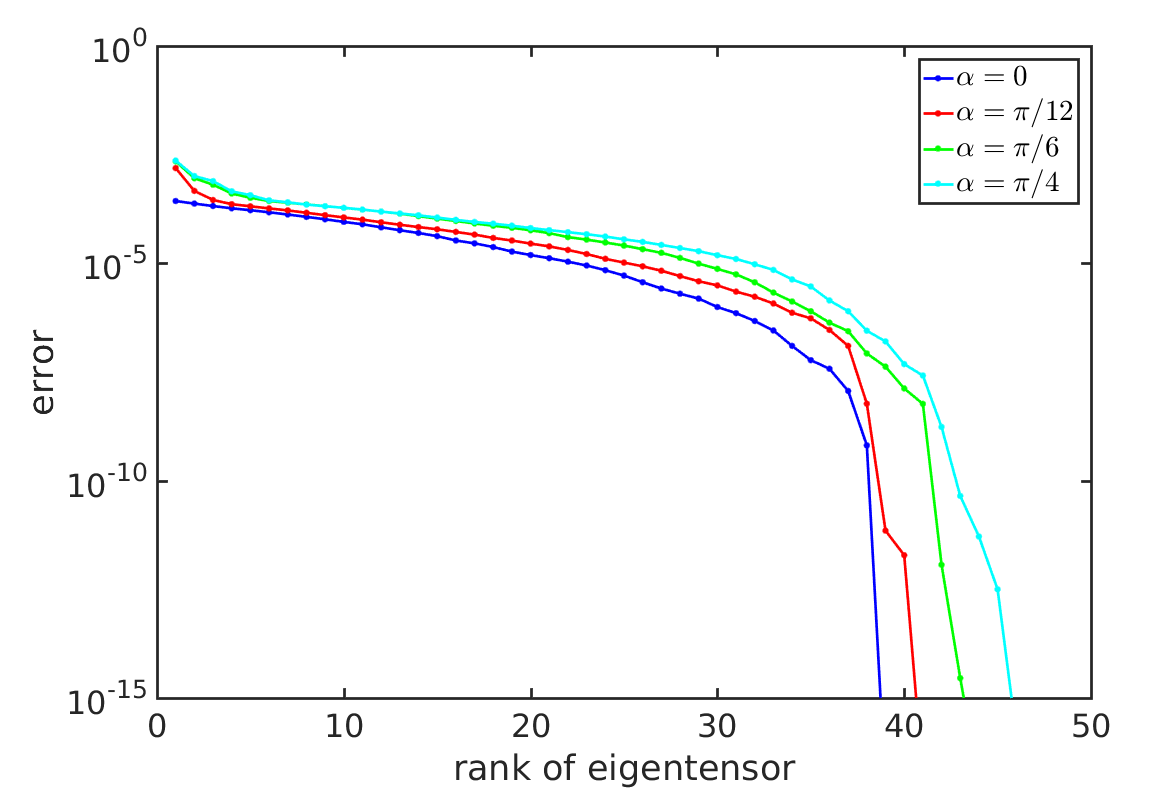}
    \end{minipage}
    \begin{minipage}[b]{0.49\textwidth}
        \centering
        \subfiguretitle{d)}
        \includegraphics[width=\textwidth]{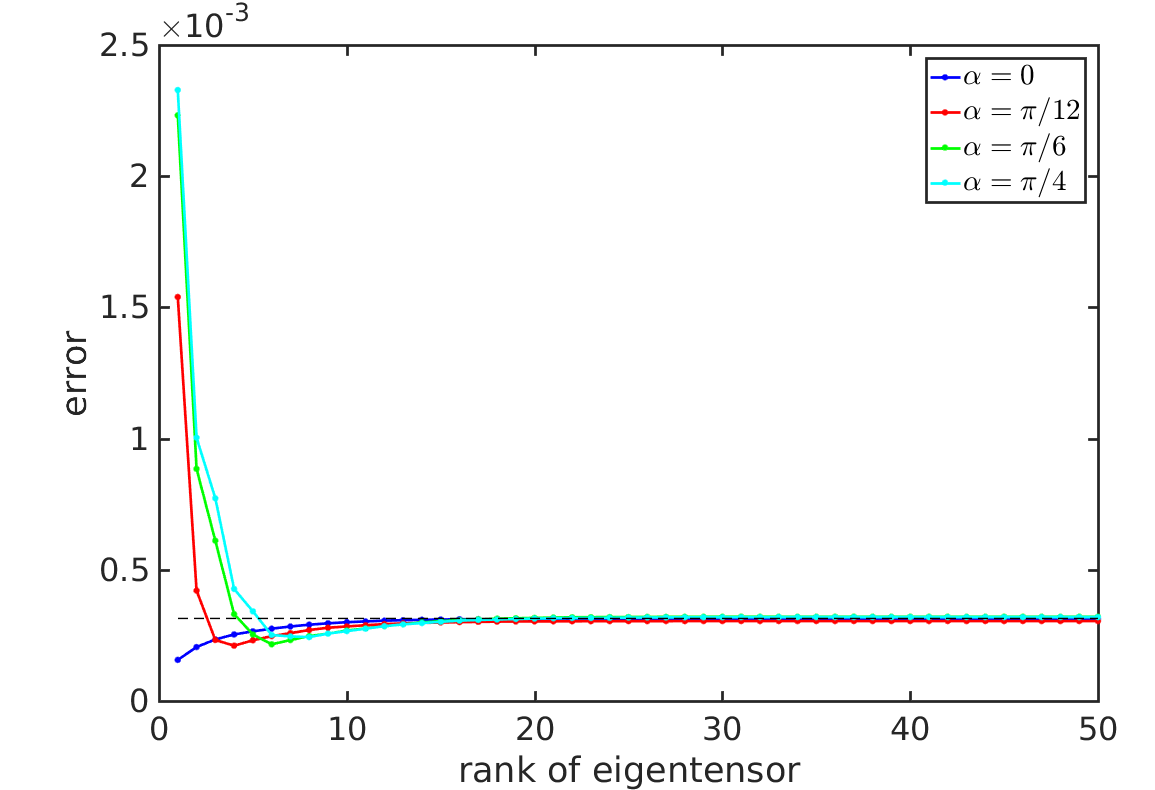}
    \end{minipage}
    \caption{(Top) Let $ \hat{v}_1 $ denote the first eigenvector of $ \hat{P} $, $ v_1 $ the (vectorized) eigentensor of the truncated tensor representation $ \mathbf{P} $, and $ \mu_\textsub{inv} $ the analytically computed invariant density. The error here is defined by $ e = \frac{1}{k} \norm{v_1 - \hat{v}_1}_2 $ and $ e = \frac{1}{k} \norm{v_1 - \mu_\textsub{inv}}_2 $, respectively, where $ k $ is the number of boxes. a) Difference between $ v_1 $ and $ \hat{v}_1 $ depending on the rank of $ \mathbf{P} $. b) Difference between $ v_1 $ and $ \mu_\textsub{inv} $. (Bottom) Influence of the truncation of the first eigentensor $ v_1 $ of the full tensor representation $ \mathbf{P} $ on the accuracy. c) Difference between the truncated eigentensor $ v_1 $ and $ \hat{v_1} $. d) Difference between the truncated eigentensor $ v_1 $ and $ \mu_\textsub{inv} $. The dashed lines show the error for the full-rank approximation which is almost identical for the different values of $ \alpha $.}
    \label{fig:Double well truncation}
\end{figure}

This example shows that in order to be able to approximate eigenfunctions by low-rank tensors, the dynamics of the system should be aligned with the axes chosen, although even if the dynamics are not aligned, the tensor format might be advantageous. In general, the dynamics are unknown a priori and not necessarily aligned with the axes, but for higher-dimensional systems it is often possible to decompose a system into slow and fast subsystems. Not all variables of a system might be equally important to describe the system's behavior. The intuition would be that certain subsystems require less information and that the tensor approximation automatically captures the relevant dynamics, using high ranks only when necessary.

\subsection{3-dimensional triple well problem}

Let us consider a more complex 3-dimensional example with the potential function
\begin{equation*}
    V(x_1, x_2, x_3) = 3 e^{-x_1^2 - (x_2 - \frac{1}{3})^2} - 3 e^{-x_1^2 - (x_2 - \frac{5}{3})^2}
                     - 5 e^{-(x_1-1)^2-x_2^2}             - 5 e^{-(x_1+1)^2 - x_2^2}
                     + \tfrac{2}{10} x_1^4 + \tfrac{2}{10} \left(x_2 - \tfrac{1}{3}\right)^4 
                    + x_3^2.
\end{equation*}
This is a potential taken from~\cite{SS13}, augmented by a third dimension. We subdivided the domain $ \mathcal{X} = [-2, 2] \times [-1, 2] \times [-2, 2] $ into $ 20 \times 20 \times 20 $ boxes of the same size. For each box, we randomly generated 1000 test points. The resulting finite-dimensional approximation is then a tensor $ \mathbf{P} \in \R^{20 \times 20 \times 20 \times 20 \times 20 \times 20} $. Figure~\ref{fig:Triple well P123} shows a scatter plot of the first three eigenfunctions, which were computed using the inverse power iteration described in Section~\ref{sec:Eigenvalue problems}. The first eigenfunction clearly shows the three expected regions with high probabilities corresponding to the minima of the potential $ V $. The second eigenfunction separates the two deeper wells at $ (-1, 0, 0 ) $ and $ (1, 0, 0) $, the third eigenfunction separates these wells from the third shallower well at $ (0, 1.5, 0) $. The differences between the eigenfunctions computed using the conventional and the tensor-based version of Ulam's method are negligible, the average difference between the first eigenvector $ v_1 $ and the tensor $ \mathbf{v}_1 $ is of the order of $ 10^{-6} $, which is mainly due to the less accurate power iteration method applied to the tensor eigenvalue problem.

\begin{figure}[htb]
    \newcommand*{\factor}{0.32}
    \centering
    \includegraphics[width=\factor\textwidth]{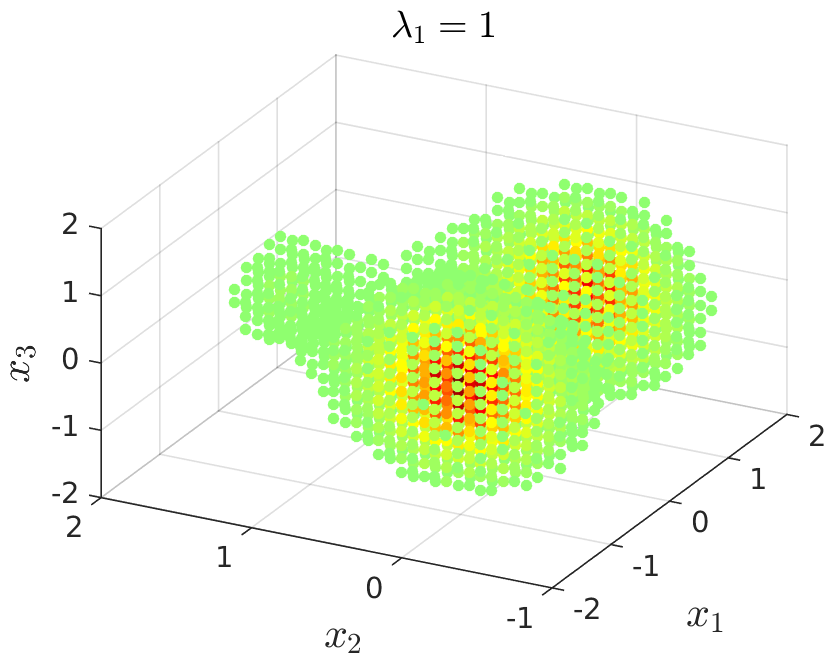}
    \includegraphics[width=\factor\textwidth]{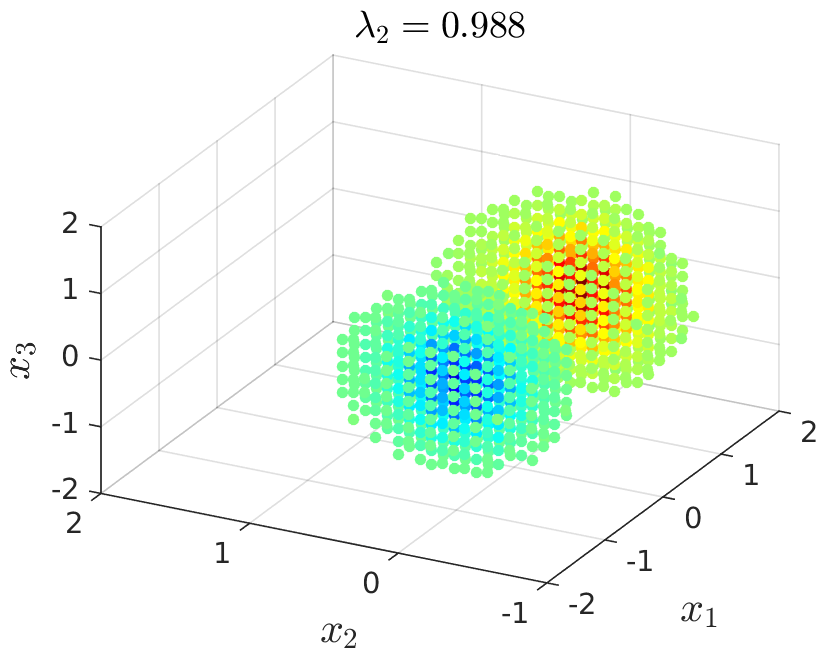}
    \includegraphics[width=\factor\textwidth]{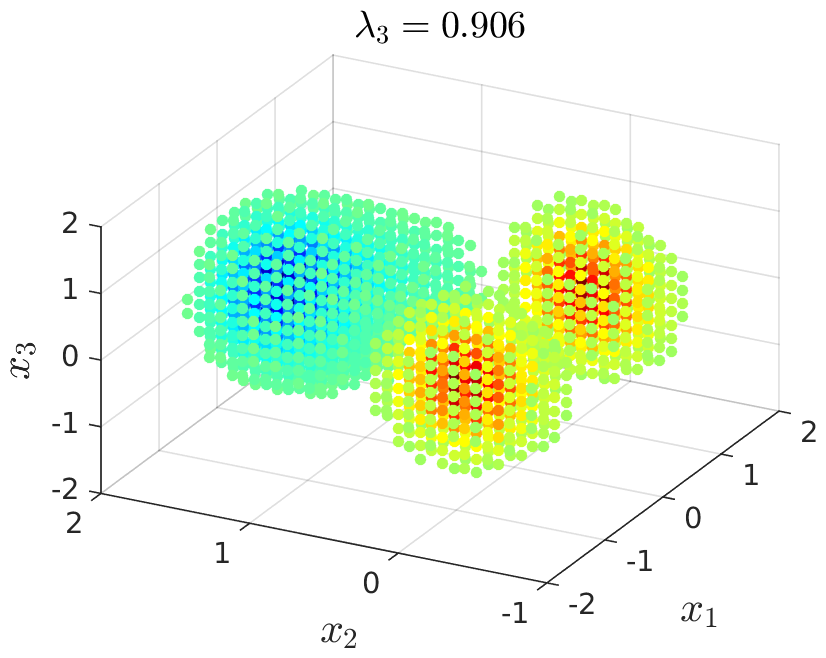}
    \caption{Scatter plot of the first three eigenfunctions of the Perron--Frobenius operator for the triple well system. Only entries whose absolute value is larger than a given threshold are plotted, entries close to zero are omitted.}
    \label{fig:Triple well P123}
\end{figure}

For the sake of comparison, we computed the eigenfunctions of the Koopman operator using EDMD. We chose basis functions $ \mathcal{D} = \{ x_1^{i_1} x_2^{i_2} x_3^{i_3}, \, i_1, i_2, i_3 = 0, \dots, 5 \} $. Hence, $ \mathbf{A}, \mathbf{G} \in \R^{6 \times 6 \times 6 \times 6 \times 6 \times 6} $. For higher-order monomials, the resulting matrices are ill-conditioned and the eigenfunctions cannot be computed accurately anymore. The results are shown in Figure~\ref{fig:Triple well K23}. As before, the second eigenfunction separates the two deeper wells. The third eigenfunction separates these wells from the third shallower well and is close to zero for the regions around the deep wells. The trivial eigenfunction of the Koopman operator corresponding to $ \lambda_1 = 1 $ is not plotted here since it is constant and does not contain relevant information about the system. Note that the eigenvalues are slightly different due to the different set of basis functions used for EDMD.

\begin{figure}[htb]
    \newcommand*{\factor}{0.32}
    \centering
    \includegraphics[width=\factor\textwidth]{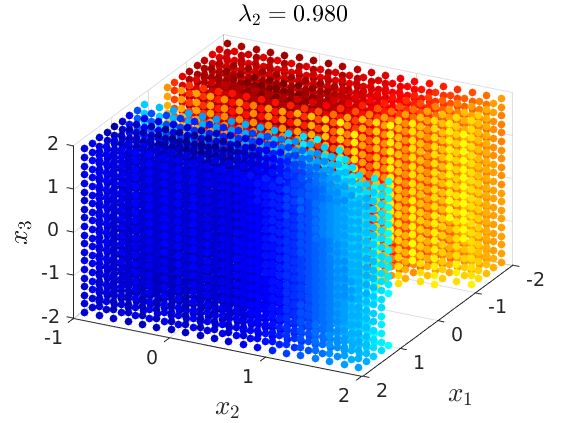}
    \includegraphics[width=\factor\textwidth]{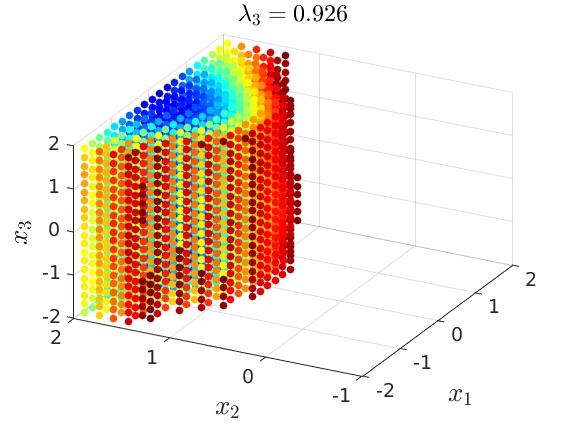}
    \caption{Scatter plot of the second and third eigenfunction of the Koopman operator for the triple well system. Note that compared to the other plots the second eigenfunction of the Koopman operator is rotated by 180 degrees around the $ x_3 $ axis for a better visualization.}
    \label{fig:Triple well K23}
\end{figure}

\section{Conclusion}
\label{sec:Conclusion}

We have reformulated the problems of computing finite-dimensional approximations of the Perron--Frobenius and Koopman operator in a different format using tensors instead of vectors. The matrices $ \mathbf{P} $ (if Ulam's method is used) or $ \mathbf{A} $ and $ \mathbf{G} $ (if EDMD is used) can now either be assembled in the dense tensor format or directly in the canonical tensor format -- which can then easily be converted into the TT format --, enabling low-rank approximations of the aforementioned operators. The next step is to systematically develop the numerical methods required to efficiently solve the resulting nonsymmetric generalized tensor eigenvalue problems and also to store and handle these tensors minimizing memory requirements so that even high-dimensional problems can be solved. First results obtained by applying simple algorithms such as power iteration methods are promising and show that the approaches presented within this paper might be able to tackle high-dimensional problems. Currently, several toolboxes for tensor-based problems are under development. Once these toolboxes contain methods for solving nonsymmetric generalized eigenvalue problems, the proposed approaches can be implemented easily, potentially facilitating the computation of meta-stable sets or almost invariant sets of dynamical systems that could not be handled before due to the curse of dimensionality. We demonstrated the tensor-based version of Ulam's method and EDMD using two- and three-dimensional problems mainly because the results can be easily validated and visualized.

Future work also includes determining which tensor format is suited best for our purposes. Currently, one of the main bottlenecks is the simulation required to obtain the data. Even simple molecular dynamics simulations, for instance, might easily take several days, but in order to capture the behavior of the system, long trajectories or a large number of short simulations with different initial conditions are required. This huge amount of data must then be processed. Thus, also the construction of the matrices $ \mathbf{P} $ or $ \mathbf{A} $ and $ \mathbf{G} $ is time-consuming. The number of boxes or basis functions required to represent each variable of the system accurately is in general unknown a priori. If we are, for instance, only interested in the leading eigenfunctions of the Koopman operator, the fast variables of the system are typically almost constant and require less information to be captured. Starting with a set of only a few basis functions for each unknown, which is then, if needed, augmented adaptively based on the system's behavior would greatly improve the efficiency. Adaptive methods combined with (sparse) tensor approaches might be able to tackle high-dimensional systems and diminish the curse of dimensionality. Furthermore, a detailed numerical analysis of the efficiency and accuracy of the proposed algorithms would help understand the limitations and find opportunities for improvement.

Another open problem is the -- depending on the number of dimensions $ d $ -- typically extremely large condition number of the matrices $ \mathbf{A} $ or $ \mathbf{G} $ if EDMD with, for instance, monomials are used to compute the eigenfunctions of the Perron--Frobenius or Koopman operator. Hence, the resulting eigenvalue problems cannot be solved accurately anymore for high-dimensional systems. A detailed understanding and numerical analysis of different basis functions might help mitigate this problem. Radial-basis functions or other more locally defined functions could lead to better results.

\section*{Acknowledgements}

This research has been partially funded by Deutsche Forschungsgemeinschaft (DFG) through grant CRC 1114. Moreover, we would like to thank the reviewers for their helpful comments and suggestions.

\bibliographystyle{unsrt}
\bibliography{TPFK}

\end{document}